\documentclass[reqno,11pt]{amsart}
\usepackage{amsmath}
\usepackage{amsxtra}
\usepackage{amscd}
\usepackage{amsthm}
\usepackage{amsfonts}
\usepackage{amssymb}
\usepackage{eucal}
\usepackage{scalerel}
\usepackage{verbatim}
\usepackage[matrix,arrow,curve]{xy}

\usepackage{a4wide}

\usepackage[T1]{fontenc}
\usepackage{xr-hyper}
\usepackage[
pdftex,
bookmarks=false,
colorlinks=true,
debug=true,
pdfnewwindow=true]{hyperref}

\theoremstyle{plain}
\newtheorem{Thm}[equation]{Theorem}
\newtheorem{Cor}[equation]{Corollary}
\newtheorem{Lem}[equation]{Lemma}
\newtheorem{Prop}[equation]{Proposition}
\newtheorem{Conj}[equation]{Conjecture}

\theoremstyle{definition}
\newtheorem{Def}[equation]{Definition}

\theoremstyle{remark}

\newtheorem{Rem}[equation]{Remark}

\newtheorem{conj}{Conjecture}

\numberwithin{equation}{section}
\renewcommand{\rm}{\normalshape}

\newif\ifShowLabels
\ShowLabelstrue
\newdimen\theight
\def\TeXref#1{%
    \leavevmode\vadjust{\setbox0=\hbox{{\tt
        \quad\quad  {\small \rm #1}}}%
    \theight=\ht0
    \advance\theight by \lineskip
    \kern -\theight \vbox to
    \theight{\rightline{\rlap{\box0}}%
    \vss}%
    }}%

\ShowLabelsfalse% comment this out if labels should be printed

\newenvironment{thm}[1]%
    { \begin{Thm} \label{T:#1}  \ifShowLabels \TeXref{T:#1} \fi }%
    { \end{Thm} }

\renewcommand{\th}[1]{\begin{thm}{#1} \sl }
\renewcommand{\eth}{\end{thm} }

\newenvironment{lemma}[1]%
    { \begin{Lem} \label{L:#1}  \ifShowLabels \TeXref{L:#1} \fi }%
    { \end{Lem} }
\newcommand{\lem}[1]{\begin{lemma}{#1} \sl}
\newcommand{\elem}{\end{lemma}}

\newenvironment{propos}[1]%
    { \begin{Prop} \label{P:#1}  \ifShowLabels \TeXref{P:#1} \fi }%
    { \end{Prop} }
\newcommand{\prop}[1]{\begin{propos}{#1}\sl }
\newcommand{\eprop}{\end{propos}}

\newenvironment{corol}[1]%
    { \begin{Cor} \label{C:#1}  \ifShowLabels \TeXref{C:#1} \fi }%
    { \end{Cor} }
\newcommand{\cor}[1]{\begin{corol}{#1} \sl }
\newcommand{\ecor}{\end{corol}}

\newenvironment{defeni}[1]%
    { \begin{Def} \label{D:#1}  \ifShowLabels \TeXref{D:#1} \fi }%
    { \end{Def} }
\newcommand{\defe}[1]{\begin{defeni}{#1} \sl }
\newcommand{\edefe}{\end{defeni}}

\newenvironment{remark}[1]%
    { \begin{Rem} \label{R:#1}  \ifShowLabels \TeXref{R:#1} \fi }%
    { \end{Rem} }
\newcommand{\rem}[1]{\begin{remark}{#1}}
\newcommand{\erem}{\end{remark}}

\newenvironment{conjec}[1]%
    { \begin{Conj} \label{Co:#1}  \ifShowLabels \TeXref{Co:#1} \fi }%
    { \end{Conj} }
\renewcommand{\conj}[1]{\begin{conjec}{#1} \sl }
\newcommand{\econj}{\end{conjec}}

\newcommand{\eq}[1]%
    { \ifShowLabels \TeXref{E:#1} \fi
       \begin{equation} \label{E:#1} }
\newcommand{\eeq}{ \end{equation} }

\newcommand{\prf}{ \begin{proof} }
\newcommand{\epr}{ \end{proof} }

\usepackage[dvipsnames]{xcolor}

\newcommand\nc{\newcommand}
\nc{\unl}{\underline}
\nc{\ol}{\overline}
\nc{\on}{\operatorname}
\nc{\BA}{{\mathbb{A}}}
\nc{\BC}{{\mathbb{C}}}
\nc{\BD}{{\mathbb{D}}}
\nc{\BF}{{\mathbb{F}}}
\nc{\BG}{{\mathbb{G}}}
\nc{\BM}{{\mathbb{M}}}
\nc{\BN}{{\mathbb{N}}}
\nc{\BO}{{\mathbb{O}}}
\nc{\BQ}{{\mathbb{Q}}}
\nc{\BP}{{\mathbb{P}}}
\nc{\BR}{{\mathbb{R}}}
\nc{\BZ}{{\mathbb{Z}}}
\nc{\BS}{{\mathbb{S}}}
\nc{\BK}{{\mathbb{K}}}
\nc{\BW}{{\mathbb{W}}}

\nc{\CA}{{\mathcal{A}}} \nc{\CB}{{\mathcal{B}}} \nc{\CalC}{{\mathcal
C}} \nc{\CalD}{{\mathcal D}} \nc{\CE}{{\mathcal{E}}}
\nc{\CF}{{\mathcal{F}}} \nc{\CG}{{\mathcal{G}}}
\nc{\CH}{{\mathcal{H}}} \nc{\CI}{{\mathcal{I}}}
\nc{\CK}{{\mathcal{K}}} \nc{\CL}{{\mathcal{L}}}
\nc{\CM}{{\mathcal{M}}} \nc{\CN}{{\mathcal{N}}}
\nc{\CO}{{\mathcal{O}}} \nc{\CP}{{\mathcal{P}}}
\nc{\CQ}{{\mathcal{Q}}} \nc{\CR}{{\mathcal{R}}}
\nc{\CS}{{\mathcal{S}}} \nc{\CT}{{\mathcal{T}}}
\nc{\CU}{{\mathcal{U}}} \nc{\CV}{{\mathcal{V}}}
\nc{\CW}{{\mathcal{W}}} \nc{\CX}{{\mathcal{X}}}
\nc{\CY}{{\mathcal{Y}}} \nc{\CZ}{{\mathcal{Z}}}

\nc{\fa}{{\mathfrak{a}}}
\nc{\fb}{{\mathfrak{b}}}
\nc{\fg}{{\mathfrak{g}}}
\nc{\fgl}{{\mathfrak{gl}}}
\nc{\fh}{{\mathfrak{h}}}
\nc{\fj}{{\mathfrak{j}}}
\nc{\fl}{{\mathfrak{l}}}
\nc{\fm}{{\mathfrak{m}}}
\nc{\fn}{{\mathfrak{n}}}
\nc{\fu}{{\mathfrak{u}}}
\nc{\fp}{{\mathfrak{p}}}
\nc{\frr}{{\mathfrak{r}}}
\nc{\fs}{{\mathfrak{s}}}
\nc{\ft}{{\mathfrak{t}}}
\nc{\fw}{{\mathfrak{w}}}
\nc{\fz}{{\mathfrak{z}}}

\nc{\fA}{{\mathfrak{A}}}
\nc{\fB}{{\mathfrak{B}}}
\nc{\fD}{{\mathfrak{D}}}
\nc{\fE}{{\mathfrak{E}}}
\nc{\fF}{{\mathfrak{F}}}
\nc{\fG}{{\mathfrak{G}}}
\nc{\fI}{{\mathfrak{I}}}
\nc{\fJ}{{\mathfrak{J}}}
\nc{\fK}{{\mathfrak{K}}}
\nc{\fL}{{\mathfrak{L}}}
\nc{\fM}{{\mathfrak{M}}}
\nc{\fN}{{\mathfrak{N}}}
\nc{\frP}{{\mathfrak{P}}}
\nc{\fQ}{{\mathfrak Q}}
\nc{\fR}{{\mathfrak R}}
\nc{\fS}{{\mathfrak S}}
\nc{\fT}{{\mathfrak{T}}}
\nc{\fU}{{\mathfrak{U}}}
\nc{\fW}{{\mathfrak{W}}}
\nc{\fY}{{\mathfrak{Y}}}
\nc{\fZ}{{\mathfrak{Z}}}

\nc{\ba}{{\mathbf{a}}}
\nc{\bb}{{\mathbf{b}}}
\nc{\bc}{{\mathbf{c}}}
\nc{\bd}{{\mathbf{d}}}
\nc{\be}{{\mathbf{e}}}
\nc{\bi}{{\mathbf{i}}}
\nc{\bj}{{\mathbf{j}}}
\nc{\bn}{{\mathbf{n}}}
\nc{\bp}{{\mathbf{p}}}
\nc{\bq}{{\mathbf{q}}}
\nc{\bu}{{\mathbf{u}}}
\nc{\bv}{{\mathbf{v}}}
\nc{\bw}{{\mathbf{w}}}
\nc{\bx}{{\mathbf{x}}}
\nc{\by}{{\mathbf{y}}}
\nc{\bz}{{\mathbf{z}}}

\nc{\bA}{{\mathbf{A}}}
\nc{\bB}{{\mathbf{B}}}
\nc{\bC}{{\mathbf{C}}}
\nc{\bD}{{\mathbf{D}}}
\nc{\bE}{{\mathbf{E}}}
\nc{\bI}{{\mathbf{I}}}
\nc{\bK}{{\mathbf{K}}}
\nc{\bH}{{\mathbf{H}}}
\nc{\bM}{{\mathbf{M}}}
\nc{\bN}{{\mathbf{N}}}
\nc{\bO}{{\mathbf{O}}}
\nc{\bQ}{{\mathbf Q}}
\nc{\bS}{{\mathbf{S}}}
\nc{\bT}{{\mathbf{T}}}
\nc{\bV}{{\mathbf{V}}}
\nc{\bW}{{\mathbf{W}}}
\nc{\bX}{{\mathbf{X}}}
\nc{\bP}{{\mathbf{P}}}
\nc{\bY}{{\mathbf{Y}}}
\nc{\bZ}{{\mathbf{Z}}}

\nc{\sA}{{\mathsf{A}}}
\nc{\sB}{{\mathsf{B}}}
\nc{\sC}{{\mathsf{C}}}
\nc{\sD}{{\mathsf{D}}}
\nc{\sF}{{\mathsf{F}}}
\nc{\sK}{{\mathsf{K}}}
\nc{\sM}{{\mathsf{M}}}
\nc{\sO}{{\mathsf{O}}}
\nc{\sQ}{{\mathsf{Q}}}
\nc{\sP}{{\mathsf{P}}}
\nc{\sT}{{\mathsf{T}}}
\nc{\sV}{{\mathsf{V}}}
\nc{\sW}{{\mathsf{W}}}
\nc{\sX}{{\mathsf{X}}}
\nc{\sZ}{{\mathsf{Z}}}

\nc{\sfb}{{\mathsf{b}}}
\nc{\sfc}{{\mathsf{c}}}
\nc{\sd}{{\mathsf{d}}}
\nc{\sg}{{\mathsf{g}}}
\nc{\sk}{{\mathsf{k}}}
\nc{\sfl}{{\mathsf{l}}}
\nc{\sfp}{{\mathsf{p}}}
\nc{\sr}{{\mathsf{r}}}
\nc{\st}{{\mathsf{t}}}
\nc{\sfu}{{\mathsf{u}}}
\nc{\sw}{{\mathsf{w}}}
\nc{\sz}{{\mathsf{z}}}
\nc{\sx}{{\mathsf{x}}}
\nc{\SX}{{\mathsf{X}}}
\nc{\sfv}{{\mathsf{v}}}

\nc{\bLambda}{{\boldsymbol{\Lambda}}}
\nc{\vv}{{\boldsymbol{v}}}
\nc{\Fl}{{{\mathcal F}\ell}}
\nc{\Gr}{{\on{Gr}}}
\nc{\CHH}{{\CH\!\!\CH}}
\nc{\lambdavee}{{\lambda^{\!\scriptscriptstyle\vee}}}
\nc{\alphavee}{\alpha^{\!\scriptscriptstyle\vee}}
\nc{\rhovee}{{\rho^{\!\scriptscriptstyle\vee}}}
\newcommand\iso{\,\vphantom{j^{X^2}}\smash{\overset{\sim}{\vphantom{\rule{0pt}{0.20em}}\smash{\longrightarrow}}}\,}
\nc{\oQM}{\vphantom{j^{X^2}}\smash{\overset{\circ}{\vphantom{\vstretch{0.7}{A}}\smash{\QM}}}}
\nc{\oZ}{{}^\dagger\!\vphantom{j^{X^2}}\smash{\overset{\circ}{\vphantom{\vstretch{0.7}{A}}\smash{Z}}}}
\nc{\odZ}{{}^\dagger\!\vphantom{j^{X^2}}\smash{\overset{\circ}{\vphantom{\vstretch{0.7}{A}}\smash{\mathfrak Z}}}^{c',c}}
\nc{\bdZ}{{}^\dagger\!\vphantom{j^{X^2}}\smash{\overset{\bullet}{\vphantom{\vstretch{0.7}{A}}\smash{\mathfrak Z}}}^{c',c}}
\nc{\oS}{\vphantom{j^{X^2}}\smash{\overset{\circ}{\vphantom{\vstretch{0.7}{A}}\smash{S}}}}
\nc{\buM}{\vphantom{j^{X^2}}\smash{\overset{\bullet}{\vphantom{\vstretch{0.7}{A}}\smash{M}}}}
\nc{\dW}{{}^\dagger\ol\CW{}}
\nc{\hW}{{}^\dagger\hat\CW{}}
\nc{\wW}{{}^\dagger\widetilde\CW{}}
\nc{\dZ}{{}^\dagger\!\fZ^{c',c}}
\nc{\dZc}{{}^\dagger\!\fZ^{c,c}}
\nc{\tZ}{{}^\dagger\!\tilde{Z}{}}
\nc{\hZ}{{}^\dagger\!\hat{Z}{}}

\nc{\ssl}{\mathfrak{sl}} \nc{\gl}{\mathfrak{gl}}
\nc{\wt}{\widetilde} \nc{\Sym}{\mathrm{Sym}} \nc{\Res}{\mathrm{Res}}
\nc{\sE}{{\mathsf{E}}} \nc{\bs}{{\mathbf{s}}}
\nc{\trig}{\mathrm{trig}} \nc{\rat}{\mathrm{rat}}
\nc{\sign}{\mathrm{sign}} \nc{\sL}{{\mathsf{L}}}
\nc{\fv}{{\mathfrak{v}}} \nc{\ad}{\mathrm{ad}}
\nc{\spsi}{{\mathsf{\psi}}} \nc{\sh}{{\mathsf{h}}}
\nc{\rtt}{\mathrm{rtt}} \nc{\qdet}{\mathrm{qdet}} \nc{\pt}{{\operatorname{pt}}}
\nc{\M}{\mathrm{M}} \nc{\Ker}{\mathrm{Ker}} \nc{\ssc}{\mathrm{sc}}
\nc{\loc}{\mathrm{loc}} \nc{\fra}{\mathrm{frac}}
\nc{\ddj}{\mathrm{DJ}} \nc{\End}{\mathrm{End}} \nc{\ev}{\mathrm{ev}}
\nc{\GL}{\mathrm{GL}} \nc{\SSym}{\mathrm{SSym}} \nc{\odd}{\mathrm{odd}}
\nc{\gr}{\mathrm{gr}} \nc{\Rees}{\mathrm{Rees}} \nc{\even}{\mathrm{even}}

\begin{document}
\title[Type $A$ shuffle superalgebras]
{Shuffle algebra realizations of type $A$ super Yangians and
 quantum affine superalgebras for all Cartan data}

\author[Alexander Tsymbaliuk]{Alexander Tsymbaliuk}
\address{A.T.:  Yale University, Department of Mathematics, New Haven, CT 06511, USA}
\email{sashikts@gmail.com}

\begin{abstract}
We introduce super Yangians of $\gl(V),\ssl(V)$ (in the new Drinfeld realization)
associated to all Dynkin diagrams. We show that all of them are isomorphic
to the super Yangians introduced by M.~Nazarov in~\cite{na}, by identifying them
with the corresponding RTT super Yangians. However, their ``positive halves'' are not
pairwise isomorphic, and we obtain the shuffle algebra realizations for all of those
in spirit of~\cite{t}. We adapt the latter to the trigonometric setup by obtaining
the shuffle algebra realizations of the ``positive halves'' of type $A$ quantum loop
superalgebras associated to arbitrary Dynkin diagrams.
\end{abstract}
\maketitle

    %%%%%%%%%%%%%%%%%%%%%%%%%%%%%%%%%%%%%%%%%%%%%%%%%%%%%%%%%%%%%%%%%%%%%%%%%%%%%%%
    %%%%%%%%%%%%%%%%%%%%%%%%%%%%%%%%%%%%%%%%%%%%%%%%%%%%%%%%%%%%%%%%%%%%%%%%%%%%%%%
    %%%%%%%%%%%%%%%%%%%%%%%%%%%%%% INTRODUCTION %%%%%%%%%%%%%%%%%%%%%%%%%%%%%%%%%%%
    %%%%%%%%%%%%%%%%%%%%%%%%%%%%%%%%%%%%%%%%%%%%%%%%%%%%%%%%%%%%%%%%%%%%%%%%%%%%%%%
    %%%%%%%%%%%%%%%%%%%%%%%%%%%%%%%%%%%%%%%%%%%%%%%%%%%%%%%%%%%%%%%%%%%%%%%%%%%%%%%

\section{Introduction}

    %%%%%%%%%%%%%%%%%%%%%%%%%%%%%%%%%%%%%%%%%%%%%%%%%%%%%%%%%%%%%%%%%%%%%%%%%%%%%%%
    %%%%%%%%%%%%%%%%%%%%%%%%%%%%%%%%%%%% SUMMARY %%%%%%%%%%%%%%%%%%%%%%%%%%%%%%%%%%
    %%%%%%%%%%%%%%%%%%%%%%%%%%%%%%%%%%%%%%%%%%%%%%%%%%%%%%%%%%%%%%%%%%%%%%%%%%%%%%%

\subsection{Summary}
\

Recall that a novel feature of Lie superalgebras (in contrast to Lie algebras) is that
they admit several non-isomorphic Dynkin diagrams.
The isomorphism of the Lie superalgebras corresponding to different Dynkin diagrams
of finite/affine type has been obtained by V.~Serganova in the Appendix to~\cite{lss}.
Likewise, one may define various quantizations of universal enveloping superalgebras
starting from different Dynkin diagrams, and establishing their isomorphism is a non-trivial
question. In the case of quantum finite/affine superalgebras, this question has been
addressed $20$ years ago by H.~Yamane~\cite{y}. However, an answer to a similar question
for super Yangians seems to be missing in the literature.

In this short note, we study type $A$ super Yangians and their shuffle realizations.
We define those in the Drinfeld realization, generalizing the construction of~\cite{s}
for a distinguished Dynkin diagram. Following~\cite{bk, g2, p2}, we obtain their RTT
realization and thus identify all of them with the super Yangian of M.~Nazarov~\cite{na}.
We also describe their centers, following~\cite{g2}.

However, the ``positive halves'' of these algebras, denoted by $Y^+_\hbar$,
do essentially depend on the choice of Dynkin diagrams. In the second part of
this note, we obtain the shuffle algebra realizations of all such $Y^+_\hbar$
and their Drinfeld-Gavarini dual $\bY^+_\hbar$. We also establish
the trigonometric (aka $q$-deformed) counterparts of these results.

This note is a companion to~\cite{t}
(the shuffle realizations were announced in~\cite[\S8.2]{t}).

    %%%%%%%%%%%%%%%%%%%%%%%%%%%%%%%%%%%%%%%%%%%%%%%%%%%%%%%%%%%%%%%%%%%%%%%%%%%%%%%
    %%%%%%%%%%%%%%%%%%%%%%%%%%%%%%%% Paper's Outline %%%%%%%%%%%%%%%%%%%%%%%%%%%%%%
    %%%%%%%%%%%%%%%%%%%%%%%%%%%%%%%%%%%%%%%%%%%%%%%%%%%%%%%%%%%%%%%%%%%%%%%%%%%%%%%

\subsection{Outline of the paper}
\

$\bullet$
In Section~\ref{ssec Drinfeld super Yangian gl}, we introduce the Drinfeld super
Yangians $Y(\gl(V))$ associated with arbitrary Dynkin diagrams of $\gl(V)$, where
$V=V_{\bar{0}}\oplus V_{\bar{1}}$ is a finite-dimensional superspace.
For the distinguished Dynkin diagram, this recovers the super Yangians
$Y_{m|n}(1)$ of~\cite{na}, due to~\cite{g2}, while for a general Dynkin diagram,
this recovers the construction of~\cite{p2}. The key result of this section,
Theorem~\ref{Drinfeld super yangians gl are iso}, establishes that $Y(\gl(V))$
is independent (up to isomorphisms) of the choice of Dynkin diagrams. The latter
may be viewed as a rational counterpart of a similar statement for quantum affine
superalgebras, due to~\cite{y}, see Remark~\ref{isomorphism of quantum affine}.

Our proof of Theorem~\ref{Drinfeld super yangians gl are iso} is crucially based
on the identification of $Y(\gl(V))$ with the RTT super Yangians $Y^\rtt(\gl(V))$
introduced in Section~\ref{ssec RTT super Yangian gl}, see
Theorem~\ref{Yangian Gauss decomposition} and Lemma~\ref{isomorphism of RTT yangians}.

In Section~\ref{ssec RTT super Yangian sl}, we introduce the RTT super Yangians
$Y^\rtt(\ssl(V))$ following the classical approach of~\cite{mno}. For
$\dim(V_{\bar{0}})\ne \dim(V_{\bar{1}})$, we obtain a decomposition
$Y^\rtt(\gl(V))\simeq Y^\rtt(\ssl(V))\otimes ZY^\rtt(\gl(V))$,
Theorem~\ref{gl vs sl and center}(a), similar to~\cite{mno}. Here,
$ZY^\rtt(\gl(V))$ denotes the center of $Y^\rtt(\gl(V))$, which is a polynomial
algebra in the coefficients of the quantum Berezinian $b(z)$ defined
in~(\ref{Berezinian}), Theorem~\ref{center of super Yangian} (for the distinguished
Dynkin diagram, $b(z)$ coincides with the quantum Berezinian of~\cite{na}, due
to~\cite[Theorem 1]{g1}). In contrast, $ZY^\rtt(\gl(V))\subset Y^\rtt(\ssl(V))$
if $\dim(V_{\bar{0}})=\dim(V_{\bar{1}})$, Theorem~\ref{gl vs sl and center}(b),
and we introduce the RTT super Yangian $Y^\rtt(A(V))$ as the corresponding
central reduction of $Y^\rtt(\ssl(V))$.

In Section~\ref{ssec super Yangian sl}, we introduce the Drinfeld super Yangians $Y(\ssl(V))$
associated with arbitrary Dynkin diagrams of $\gl(V)$, and construct superalgebra embeddings
$Y(\ssl(V))\hookrightarrow Y(\gl(V))$ and isomorphisms $Y(\ssl(V))\iso Y^\rtt(\ssl(V))$,
Theorem~\ref{sl vs gl Drinfeld}. The latter implies that super Yangians $Y(\ssl(V))$
associated with various Dynkin diagrams are pairwise isomorphic,
Theorem~\ref{isomorphism of Drinfeld yangians sl}.

In Section~\ref{ssec triangular decomposition}, we recall the PBW theorem and
the triangular decomposition for $Y(\ssl(V))$, Theorem~\ref{PBW for Yangian}
and Proposition~\ref{Triangular decomposition}.

In Section~\ref{ssec formal Yangian}, we introduce a $\BC[\hbar]$-version
$Y_\hbar(\ssl(V))$ and its Drinfeld-Gavarini dual subalgebra $\bY_\hbar(\ssl(V))$.
They can be viewed as Rees algebras~(\ref{Rees algebras}) of $Y(\ssl(V))$ with respect
to two standard filtrations on it defined via~(\ref{Gradings}), Remark~\ref{formal vs nonformal}.
The PBW Theorems for $Y_\hbar(\ssl(V))$ and $\bY_\hbar(\ssl(V))$,
Theorems~\ref{PBW for formal Yangian} and~\ref{pbw for entire gavarini yangian},
follow from~\cite[Theorem B.3, Theorem A.7]{ft2}.

$\bullet$
In Section~\ref{ssec super shuffle algebra}, we introduce the rational shuffle
(super)algebra $\bar{W}^V$, Definition~\ref{Definition of Shuffle}, which may be
viewed as a rational super counterpart of the elliptic shuffle algebras of
Feigin-Odesskii, \cite{fo1}--\cite{fo3}. It is related to the ``positive half''
$Y^+_\hbar(\ssl(V))$ of the super Yangian $Y_\hbar(\ssl(V))$ via an explicit homomorphism
$\Psi\colon Y^+_\hbar(\ssl(V))\to \bar{W}^{V}$ of Proposition~\ref{simple shuffle}.
The injectivity of $\Psi$ is established in Corollary~\ref{Injectivity of Psi}.
The key results of Section~\ref{sec shuffle for superYangians} describe the images
of $Y^+_\hbar(\ssl(V))$ and its Drinfeld-Gavarini dual subalgebra $\bY^+_\hbar(\ssl(V))$
under $\Psi$, Theorems~\ref{hard shuffle yangian},~\ref{shuffle integral form yangian}.
The latter is used to obtain a new proof of the PBW property for $\bY^+_\hbar(\ssl(V))$,
Theorem~\ref{pbw for gavarini yangian}.

In Section~\ref{ssec rank 1 case}, we establish the key result
in the simplest case $\dim(V)=2$, Theorem~\ref{rank 1 iso}.

In Section~\ref{ssec usual shuffle algebra}, we introduce our key technical
tool in the study of the shuffle algebras, the \emph{specialization maps}
$\phi_{\unl{d}}$~(\ref{specialization map}). Their two main properties are
established in Lemmas~\ref{lower degrees},~\ref{same degrees}, which immediately
imply the injectivity of $\Psi$, Corollary~\ref{Injectivity of Psi}.

In Section~\ref{ssec image description}, we finally describe the images of
$Y^+_\hbar(\ssl(V))$ and its subalgebra $\bY^+_\hbar(\ssl(V))$,
Theorems~\ref{hard shuffle yangian},~\ref{shuffle integral form yangian}.
The former consists of all \emph{good} shuffle elements,
Definition~\ref{good element yangian}, while the latter consists of all
\emph{integral} shuffle elements, Definition~\ref{integral element yangian}.
We also prove Theorem~\ref{pbw for gavarini yangian}.

$\bullet$
In Section~\ref{sec shuffle for superquantum}, we recall the definition of
$U^>_{\vv}(L\gl(V))$, the ``positive half'' of the quantum loop superalgebra of $\gl(V)$,
and obtain its shuffle algebra realization, Theorem~\ref{quantum shuffle isom}.
This provides the trigonometric counterpart of Theorem~\ref{hard shuffle yangian}
and generalizes~\cite[Theorem 5.17]{t}, where this result was established for
the distinguished Dynkin diagram of $\gl(V)$.

    %%%%%%%%%%%%%%%%%%%%%%%%%%%%%%%%%%%%%%%%%%%%%%%%%%%%%%%%%%%%%%%%%%%%%%%%%%%%%%%
    %%%%%%%%%%%%%%%%%%%%%%%%%%%%%%%% ACKNOWLEDGMENTS %%%%%%%%%%%%%%%%%%%%%%%%%%%%%%
    %%%%%%%%%%%%%%%%%%%%%%%%%%%%%%%%%%%%%%%%%%%%%%%%%%%%%%%%%%%%%%%%%%%%%%%%%%%%%%%

\subsection{Acknowledgments}
\

I am indebted to Boris Feigin, Michael Finkelberg, and Andrei Negu\c{t} for
numerous interesting discussions over the years; to Luan Bezerra and Evgeny Mukhin
for a useful correspondence on quantum affine superalgebras; to Vasily Pestun for
his invitation to IHES; to Yung-Ning Peng for bringing~\cite{p1, p2} (related to our
Sections~\ref{ssec Drinfeld super Yangian gl},~\ref{ssec RTT super Yangian gl})
to my attention after the first version of this note was posted on arXiv;
to anonymous referees for many useful suggestions.

This note was prepared during the author's visit to IHES (Bures-sur-Yvette, France)
in the summer $2019$, sponsored by the European Research Council (ERC) under the European
Union's Horizon $2020$ research and innovation program (QUASIFT grant agreement $677368$).

I gratefully acknowledge the NSF Grant DMS-1821185.

    %%%%%%%%%%%%%%%%%%%%%%%%%%%%%%%%%%%%%%%%%%%%%%%%%%%%%%%%%%%%%%%%%%%%%%%%%%%%%%%
    %%%%%%%%%%%%%%%%%%%%%%%%%%%%%%%%%%%%%%%%%%%%%%%%%%%%%%%%%%%%%%%%%%%%%%%%%%%%%%%
    %%%%%%%%%%%%%%%%%%%%%%%%%%%%%%%% Super Yangians %%%%%%%%%%%%%%%%%%%%%%%%%%%%%%%
    %%%%%%%%%%%%%%%%%%%%%%%%%%%%%%%%%%%%%%%%%%%%%%%%%%%%%%%%%%%%%%%%%%%%%%%%%%%%%%%
    %%%%%%%%%%%%%%%%%%%%%%%%%%%%%%%%%%%%%%%%%%%%%%%%%%%%%%%%%%%%%%%%%%%%%%%%%%%%%%%

\section{Type $A$ super Yangians}\label{sec super Yangian}

    %%%%%%%%%%%%%%%%%%%%%%%%%%%%%%%%%%%%%%%%%%%%%%%%%%%%%%%%%%%%%%%%%%%%%%%%%%%%%%%
    %%%%%%%%%%%%%%%%%%%%%%%%%%%%%%%%%%%%%%% Setup %%%%%%%%%%%%%%%%%%%%%%%%%%%%%%%%%
    %%%%%%%%%%%%%%%%%%%%%%%%%%%%%%%%%%%%%%%%%%%%%%%%%%%%%%%%%%%%%%%%%%%%%%%%%%%%%%%

\subsection{Setup and notations}\label{ssec setup}
\

Consider a superspace $V=V_{\bar{0}}\oplus V_{\bar{1}}$ with a $\BC$-basis
$\sfv_1,\ldots,\sfv_n$ such that each $\sfv_i$ is either \emph{even}
($\sfv_i\in V_{\bar{0}}$) or \emph{odd} ($\sfv_i\in V_{\bar{1}}$).
We set $n_{+}:=\dim(V_{\bar{0}}), n_{-}:=\dim(V_{\bar{1}}), n:=n_++n_-=\dim(V)$.
For $1\leq i\leq n$, define $\ol{i}\in \BZ_2$ via
  $\ol{i}=\begin{cases}
     \bar{0}, & \text{if } \sfv_i\in V_{\bar{0}}\\
     \bar{1}, & \text{if } \sfv_i\in V_{\bar{1}}
   \end{cases}$.
Consider a free $\BZ$-module $P:=\oplus_{i=1}^{n} \BZ\epsilon_i$ with the
bilinear form determined by $(\epsilon_i,\epsilon_j)=\delta_{i,j}(-1)^{\ol{i}}$
(we set $(-1)^{\bar{0}}:=1, (-1)^{\bar{1}}:=-1$).
For $1\leq i<n$, let $\alpha_i:=\epsilon_{i}-\epsilon_{i+1}\in P$ be the
simple roots of $\gl(V)$, and
  $\Delta^+:=\{\epsilon_j-\epsilon_i\}_{1\leq j<i\leq n}\subset P$
be the set of positive roots of $\gl(V)$. Let $I=\{1,2,\ldots,n-1\}$
and set $|\alpha_i|:=\ol{i}+\ol{i+1}\in \BZ_2$ for $i\in I$. Finally,
let $(c_{ij})_{i,j\in I}$  be the associated Cartan matrix, that is,
$c_{ij}:=(\alpha_i,\alpha_j)$.

For a superalgebra $A$ and its two homogeneous elements $x$ and $x'$, we define
\begin{equation}\label{super commutator and anticommutator}
  [x,x']:=xx'-(-1)^{|x|\cdot |x'|}x'x\ \ \mathrm{and}\ \
  \{x,x'\}:=xx'+(-1)^{|x|\cdot |x'|}x'x,
\end{equation}
where $|x|$ denotes the $\BZ_2$-grading of $x$ (that is, $x\in A_{|x|}$).

Given two superspaces $A=A_{\bar{0}}\oplus A_{\bar{1}}$ and
$B=B_{\bar{0}}\oplus B_{\bar{1}}$, their tensor product $A\otimes B$
is also a superspace with
  $(A\otimes B)_{\bar{0}}=A_{\bar{0}}\otimes B_{\bar{0}}\oplus A_{\bar{1}}\otimes B_{\bar{1}}$
and
  $(A\otimes B)_{\bar{1}}=A_{\bar{0}}\otimes B_{\bar{1}}\oplus A_{\bar{1}}\otimes B_{\bar{0}}$.
Furthermore, if $A$ and $B$ are superalgebras, then $A\otimes B$ is made
into a superalgebra, the \emph{graded tensor product} of the superalgebras
$A$ and $B$, via the following multiplication:
\begin{equation}\label{graded tensor product}
  (x\otimes y)(x'\otimes y')=(-1)^{|y|\cdot |x'|} (xx')\otimes (yy')
  \ \mathrm{for\ any}\ x\in A_{|x|}, x'\in A_{|x'|}, y\in B_{|y|},y'\in B_{|y'|}.
\end{equation}
We will use only graded tensor products of superalgebras throughout this paper.

    %%%%%%%%%%%%%%%%%%%%%%%%%%%%%%%%%%%%%%%%%%%%%%%%%%%%%%%%%%%%%%%%%%%%%%%%%%%%%%%
    %%%%%%%%%%%%%%%%%%%%%%%%%%%% Drinfeld Super Yangian gl %%%%%%%%%%%%%%%%%%%%%%%%
    %%%%%%%%%%%%%%%%%%%%%%%%%%%%%%%%%%%%%%%%%%%%%%%%%%%%%%%%%%%%%%%%%%%%%%%%%%%%%%%

\subsection{The Drinfeld super Yangian of $\gl(V)$}\label{ssec Drinfeld super Yangian gl}
\

Following~\cite{d,bk,g2,p2}, define the \emph{Drinfeld super Yangian of $\gl(V)$},
denoted by $Y(\gl(V))$, to be the associative $\BC$-superalgebra generated by
  $\{d_j^{(s)}, \wt{d}_j^{(s)}, e_i^{(r)}, f_i^{(r)}\}_{1\leq i<n, 1\leq j\leq n}^{r\geq 1,s\geq 0}$
with the $\BZ_2$-grading
  $|d_j^{(r)}|=|\wt{d}_j^{(r)}|=\bar{0},\ |e^{(r)}_i|=|f^{(r)}_i|=|\alpha_i|$,
and subject to the following defining relations:
\begin{equation}\label{Dr Yangian 0}
  d^{(0)}_j=1,\ \wt{d}^{(0)}_j=1,\
  \sum_{t=0}^r \wt{d}^{(t)}_i d^{(r-t)}_i=\delta_{r,0},
\end{equation}
\begin{equation}\label{Dr Yangian 1}
  [d^{(r)}_i,d^{(s)}_{j}]=0,
\end{equation}
\begin{equation}\label{Dr Yangian 2}
  [d^{(r)}_i,e^{(s)}_{j}]=
  (-1)^{\ol{i}}(\delta_{i,j}-\delta_{i,j+1}) \sum_{t=0}^{r-1} d^{(t)}_i e^{(r+s-t-1)}_j,
\end{equation}
\begin{equation}\label{Dr Yangian 3}
  [d^{(r)}_i,f^{(s)}_{j}]=
  (-1)^{\ol{i}}(-\delta_{i,j}+\delta_{i,j+1}) \sum_{t=0}^{r-1} f^{(r+s-t-1)}_j d^{(t)}_i,
\end{equation}
\begin{equation}\label{Dr Yangian 4}
  [e_{i}^{(r)},f_{j}^{(s)}]=
  -(-1)^{\ol{i+1}}\delta_{i,j} \sum_{t=0}^{r+s-1} \wt{d}^{(t)}_i d^{(r+s-t-1)}_{i+1},
\end{equation}
\begin{equation}\label{Dr Yangian 5}
\begin{split}
  & [e^{(r)}_i,e^{(s)}_j]=0 \ \mathrm{if}\ c_{ij}=0,\\
  & [e^{(r)}_i, e^{(s+1)}_{i+1}]-[e^{(r+1)}_i,e^{(s)}_{i+1}]=
    -(-1)^{\ol{i+1}} e^{(r)}_i e^{(s)}_{i+1},\\
  & [e^{(r)}_i,e^{(s)}_i]=
    (-1)^{\ol{i}}\sum_{t=1}^{s-1} e^{(t)}_i e^{(r+s-t-1)}_i -
    (-1)^{\ol{i}}\sum_{t=1}^{r-1} e^{(t)}_i e^{(r+s-t-1)}_i
    \ \mathrm{if}\ |\alpha_i|=\bar{0},
\end{split}
\end{equation}
\begin{equation}\label{Dr Yangian 6}
\begin{split}
  & [f^{(r)}_i,f^{(s)}_j]=0 \ \mathrm{if}\ c_{ij}=0,\\
  & [f^{(s+1)}_{i+1},f^{(r)}_i]-[f^{(s)}_{i+1},f^{(r+1)}_i]=
     -(-1)^{\ol{i+1}} f^{(s)}_{i+1} f^{(r)}_i,\\
  & [f^{(r)}_i,f^{(s)}_i]=
    (-1)^{\ol{i}}\sum_{t=1}^{r-1} f^{(r+s-t-1)}_i f^{(t)}_i -
    (-1)^{\ol{i}}\sum_{t=1}^{s-1} f^{(r+s-t-1)}_i f^{(t)}_i
    \ \mathrm{if}\ |\alpha_i|=\bar{0},
%%%%%%%%%%%%%%%%%%%%%%%%%%%%%%%%%%%%%%%%%%%%%%%%%%%%%%%%%%%%%%%%%%%%%%%%%%%%%
% Note that~\cite{g2} uses the opposite order of factors in both sums.
% However, that clearly leads to the same expression on the right-hand side.
%%%%%%%%%%%%%%%%%%%%%%%%%%%%%%%%%%%%%%%%%%%%%%%%%%%%%%%%%%%%%%%%%%%%%%%%%%%%%
\end{split}
\end{equation}
as well as cubic Serre relations
\begin{equation}\label{Dr Yangian 9}
  [e^{(r)}_{i},[e^{(s)}_{i},e^{(t)}_{j}]]+[e^{(s)}_{i},[e^{(r)}_{i},e^{(t)}_{j}]]=0
  \ \mathrm{if}\ j=i\pm 1\ \mathrm{and}\ |\alpha_i|=\bar{0},
\end{equation}
\begin{equation}\label{Dr Yangian 10}
  [f^{(r)}_{i},[f^{(s)}_{i},f^{(t)}_{j}]]+[f^{(s)}_{i},[f^{(r)}_{i},f^{(t)}_{j}]]=0
  \ \mathrm{if}\ j=i\pm 1\ \mathrm{and}\ |\alpha_i|=\bar{0},
\end{equation}
and quartic Serre relations
\begin{equation}\label{Dr Yangian 11}
  [[e^{(r)}_{j-1},e^{(1)}_j],[e^{(1)}_j,e^{(s)}_{j+1}]]=0
  \ \mathrm{if}\ |\alpha_j|=\bar{1}\ \mathrm{and}\ |\alpha_{j-1}|=|\alpha_{j+1}|=\bar{0},
\end{equation}
\begin{equation}\label{Dr Yangian 12}
  [[f^{(r)}_{j-1},f^{(1)}_j],[f^{(1)}_j,f^{(s)}_{j+1}]]=0
  \ \mathrm{if}\ |\alpha_j|=\bar{1}\ \mathrm{and}\ |\alpha_{j-1}|=|\alpha_{j+1}|=\bar{0}.
\end{equation}

\begin{Rem}\label{Serre hold always}
(a) The cubic Serre relations~(\ref{Dr Yangian 9},~\ref{Dr Yangian 10})
are also valid for $|\alpha_i|=\bar{1}$, but in that case, they already
follow from $[e^{(r)}_i, e_i^{(s)}]=0=[f^{(r)}_i, f_i^{(s)}]$, due to
quadratic relations~(\ref{Dr Yangian 5},~\ref{Dr Yangian 6}).
%%%%%%%%%%%%%%%%%%%%%%%%%%%%%%%%%%%%%%%%%%%%%%%%%%%%%%%%%%%%%%%%%%%%
%%%%%%%%%%%%%%%%%%%%%%% COMMENTED PROOF %%%%%%%%%%%%%%%%%%%%%%%%%%%%
%%%%%%%%%%%%%%%%%%%%%%%%%%%%%%%%%%%%%%%%%%%%%%%%%%%%%%%%%%%%%%%%%%%%
%If $|\alpha_i|=\bar{1},\ j=i\pm 1$, and $|\alpha_j|$ is either $\bar{0}$ or $\bar{1}$, then
%\begin{multline}
%  [e_i(z_1),[e_i(z_2),e_j(w)]]+[e_i(z_1),[e_i(z_2),e_j(w)]]=\\
%  \{e_i(z_1)e_i(z_2)e_j(w)\mp e_i(z_1)e_j(w)e_i(z_2)\pm e_i(z_2)e_j(w)e_i(z_1)-e_j(w)e_i(z_2)e_i(z_1)\}+\{z_1\leftrightarrow z_2\}=\\
%  (e_i(z_1)e_i(z_2)+e_i(z_2)e_i(z_1))e_j(w)+e_j(w)(e_i(z_1)e_i(z_2)+e_i(z_2)e_i(z_1))=0.
%\end{multline}
%%%%%%%%%%%%%%%%%%%%%%%%%%%%%%%%%%%%%%%%%%%%%%%%%%%%%%%%%%%%%%%%%%%%
%%%%%%%%%%%%%%%%%%%%% END of the PROOF %%%%%%%%%%%%%%%%%%%%%%%%%%%%%
%%%%%%%%%%%%%%%%%%%%%%%%%%%%%%%%%%%%%%%%%%%%%%%%%%%%%%%%%%%%%%%%%%%%

\noindent
(b) The quartic Serre relations~(\ref{Dr Yangian 11},~\ref{Dr Yangian 12})
are also valid for any other parities of $\alpha_{j-1},\alpha_j,\alpha_{j+1}$,
but in those cases, they already follow from the quadratic and cubic
relations~(\ref{Dr Yangian 5}--\ref{Dr Yangian 10}).
%%%%%%%%%%%%%%%%%%%%%%%%%%%%%%%%%%%%%%%%%%%%%%%%%%%%%%%%%%%%%%%%%%%%
%%%%%%%%%%%%%%%%%%%%%%% COMMENTED PROOF %%%%%%%%%%%%%%%%%%%%%%%%%%%%
%%%%%%%%%%%%%%%%%%%%%%%%%%%%%%%%%%%%%%%%%%%%%%%%%%%%%%%%%%%%%%%%%%%%
%One way to prove this is to note that following our [Narabotki, page 13]
%or [Gow 2, Lemma 5], we always have these equalities on the RTT side.
%
%Another way to prove this is to verify the corresponding equality on the
%shuffle side, where it boils down to the proof of
%$[\frac{x_{j-1}^{r-1}}{x_{j-1}-x_j}, \frac{x_{j+1}^{s-1}}{x_{j+1}-x_j}]$.
%It suffices to check it for $|\alpha_{j-1}|=|\alpha_{j+1}|=\bar{0}$ and
%$|\alpha_j|$ being either $\bar{0}$ or $\bar{1}$, since varying then
%the parities of $\alpha_{j\pm 1}$ gives the same expressions.
%%%%%%%%%%%%%%%%%%%%%%%%%%%%%%%%%%%%%%%%%%%%%%%%%%%%%%%%%%%%%%%%%%%%
%%%%%%%%%%%%%%%%%%%%% END of the PROOF %%%%%%%%%%%%%%%%%%%%%%%%%%%%%
%%%%%%%%%%%%%%%%%%%%%%%%%%%%%%%%%%%%%%%%%%%%%%%%%%%%%%%%%%%%%%%%%%%%

\noindent
(c) Generalizing the quartic Serre relations~(\ref{Dr Yangian 11},~\ref{Dr Yangian 12}),
the following relations also hold:
\begin{equation}\label{Dr Yangian Serre 4 alternative 1}
  [[e^{(r)}_{j-1},e^{(k)}_j],[e^{(l)}_j,e^{(s)}_{j+1}]]+
  [[e^{(r)}_{j-1},e^{(l)}_j],[e^{(k)}_j,e^{(s)}_{j+1}]]=0,
\end{equation}
\begin{equation}\label{Dr Yangian Serre 4 alternative 2}
  [[f^{(r)}_{j-1},f^{(k)}_j],[f^{(l)}_j,f^{(s)}_{j+1}]]+
  [[f^{(r)}_{j-1},f^{(l)}_j],[f^{(k)}_j,f^{(s)}_{j+1}]]=0,
\end{equation}
cf.~Remark~\ref{Serre hold always 2}(b) and the explanations therein.
We note that these relations~(\ref{Dr Yangian Serre 4 alternative 1},~\ref{Dr Yangian Serre 4 alternative 2})
play a crucial role in the recent paper~\cite{p3}.
\end{Rem}

As pointed out to us by Y.-N.~Peng, the above definition of $Y(\gl(V))$
is actually equivalent to the one from~\cite{p1}.
In the particular case (associated with the distinguished Dynkin diagram)
\begin{equation}\label{distinguished Dynkin}
  \sfv_1,\ldots, \sfv_{n_+}\in V_{\bar{0}}
  \ \mathrm{and}\
  \sfv_{n_+ + 1},\ldots, \sfv_{n}\in V_{\bar{1}}
\end{equation}
(so that $|\alpha_{n_+}|=\bar{1}$ and $|\alpha_{i\ne n_+}|=\bar{0}$),
the defining relations~(\ref{Dr Yangian 0}--\ref{Dr Yangian 12}) first appeared
in~\cite[Theorem 3]{g2}\footnote{We note the following typos in~\cite{g2}:
$j\leq m+1$ should be replaced by $j\geq m+1$ in the third line of~(39),
the sign $(-1)^{\ol{j}}$ should be replaced by $(-1)^{\ol{j+1}}$ in
the right-hand sides of (44, 45).}, where it was shown that the corresponding
super Yangian is isomorphic to the super Yangian $Y^\rtt(\gl_{n_+|n_-})$
first introduced in~\cite{na} (thus generalizing~\cite[Theorem 5.2]{bk}).
The same arguments can be used to establish the following result
(mentioned first in~\cite{p1}):

\begin{Thm}\label{Drinfeld super yangians gl are iso}
The superalgebra $Y(\gl(V))$ depends only on $(n_+,n_-)$, up to an isomorphism.
\end{Thm}

This is a direct consequence of Theorem~\ref{Yangian Gauss decomposition} and
Lemma~\ref{isomorphism of RTT yangians}, see Remark~\ref{isomorphism of Drinfeld yangians gl}.

\begin{Rem}\label{isomorphism of quantum affine}
The quantum affine superalgebras corresponding to different Dynkin diagrams
of the same affine Lie superalgebra are known to be pairwise isomorphic, due to~\cite{y}.
A similar statement for super Yangians seems to be missing in the literature.
Thus, Theorem~\ref{Drinfeld super yangians gl are iso} and its $\ssl(V)$-counterpart,
Theorem~\ref{isomorphism of Drinfeld yangians sl}, fill this gap at least in type $A$.
\end{Rem}

    %%%%%%%%%%%%%%%%%%%%%%%%%%%%%%%%%%%%%%%%%%%%%%%%%%%%%%%%%%%%%%%%%%%%%%%%%%%%%%%
    %%%%%%%%%%%%%%%%%%%%%%%%%%%%%%% RTT Super Yangian gl %%%%%%%%%%%%%%%%%%%%%%%%%%
    %%%%%%%%%%%%%%%%%%%%%%%%%%%%%%%%%%%%%%%%%%%%%%%%%%%%%%%%%%%%%%%%%%%%%%%%%%%%%%%

\subsection{The RTT super Yangian of $\gl(V)$}\label{ssec RTT super Yangian gl}
\

Let $P\colon V\otimes V\to V\otimes V$ be the permutation operator given by
  $P:=\sum_{i,j} (-1)^{\ol{j}}E_{ij}\otimes E_{ji}$,
so that
  $P(\sfv_j\otimes \sfv_i)=(-1)^{\ol{i}\cdot \ol{j}} \sfv_i\otimes \sfv_j$.
Consider the \emph{rational} $R$-matrix
  $R_\rat(z)=1-\frac{1}{z}P\in (\End\ V)^{\otimes 2}$.

Following~\cite{frt,m,na,p1}, define the \emph{RTT super Yangian of $\gl(V)$},
denoted by $Y^\rtt(\gl(V))$, to be the associative $\BC$-superalgebra generated by
  $\{t^{(r)}_{ij}\}_{1\leq i,j\leq n}^{r\geq 1}$
with the $\BZ_2$-grading $|t_{ij}^{(r)}|=\ol{i}+\ol{j}$ and
subject to the following defining relation:
\begin{equation}\label{RTT matrix form}
  R_{\rat}(z-w)T_1(z)T_2(w)=T_2(w)T_1(z)R_\rat(z-w).
\end{equation}
Here, $T(z)$ is the series in $z^{-1}$ with coefficients in the algebra
$Y^\rtt(\gl(V))\otimes \End\ V$, defined~by
\begin{equation}\label{T matrix}
  T(z)=\sum_{i,j} (-1)^{\ol{j}(\ol{i}+1)} t_{ij}(z)\otimes E_{ij}
  \ \mathrm{with}\
  t_{ij}(z):=\delta_{i,j}+\sum_{r\geq 1} t^{(r)}_{ij}z^{-r}.
\end{equation}

\begin{Rem}\label{sign explanation}
Here, we identify the operator
  $\sum_{i,j=1}^n (-1)^{\ol{j}(\ol{i}+1)} t_{ij}(z)\otimes E_{ij}$
with the matrix $(t_{ij}(z))_{i,j=1}^n$. Evoking the
multiplication~(\ref{graded tensor product}) on the graded tensor products,
we see that the extra sign $(-1)^{\ol{j}(\ol{i}+1)}$ ensures that the product
of two matrices is calculated in the usual way.
\end{Rem}

Multiplying both sides of~(\ref{RTT matrix form}) by $z-w$,
we obtain an equality of series in $z,w$ with coefficients in
  $Y^\rtt(\gl(V))\otimes (\End\ V)^{\otimes 2}$.
Thus, relation~(\ref{RTT matrix form}) is equivalent to the following relations:
\begin{equation}\label{RTT termwise}
  (z-w)[t_{ij}(z),t_{kl}(w)]=
  (-1)^{\ol{i}\cdot \ol{j}+\ol{i}\cdot \ol{k}+\ol{j}\cdot \ol{k}}
  \left(t_{kj}(z)t_{il}(w)-t_{kj}(w)t_{il}(z)\right)
\end{equation}
for all $1\leq i,j,k,l\leq n$.

In the particular case~(\ref{distinguished Dynkin}), we recover
the super Yangian $Y^\rtt(\gl_{n_+|n_-})$ of~\cite{na} (denoted by
$Y_{n_+|n_-}(1)$ in~\cite{na}), while for a general case, we actually
get isomorphic algebras, due to the following simple result:

\begin{Lem}\label{isomorphism of RTT yangians}
The superalgebra $Y^\rtt(\gl(V))$ depends only on $(n_+,n_-)$, up to an isomorphism.
In particular, $Y^\rtt(\gl(V))$ is isomorphic to the super Yangian
$Y^\rtt(\gl_{n_+|n_-})$ of~\cite{na}.
\end{Lem}

\begin{proof}
Let $V'$ be another superspace with a $\BC$-basis $\sfv'_1,\ldots,\sfv'_n$
such that each $\sfv'_i$ is either \emph{even} or \emph{odd} and
$n'_+=n_+, n'_-=n_-$. Pick a permutation $\sigma\in \Sigma_n$, such that
$\sfv_i\in V$ and $\sfv'_{\sigma(i)}\in V'$ have the same parity for all $i$.
Then, the assignment $t^{(r)}_{ij}\mapsto t^{(r)}_{\sigma(i),\sigma(j)}$
is compatible with the defining relations~(\ref{RTT termwise}), thus
giving rise to an isomorphism $Y^\rtt(\gl(V))\iso Y^\rtt(\gl(V'))$.
\end{proof}

We also have two standard relations between $Y^\rtt(\gl(V))$ and $U(\gl(V))$
(cf.~\cite{g2,p2}):

\begin{Lem}\label{evaluation homom}
(a) The assignment $E_{ij}\mapsto (-1)^{\ol{i}}t^{(1)}_{ij}$ gives rise
to a superalgebra embedding
\begin{equation*}
  \iota\colon U(\gl(V))\hookrightarrow Y^\rtt(\gl(V)).
\end{equation*}

\noindent
(b) The assignment $t^{(r)}_{ij}\mapsto (-1)^{\ol{i}}\delta_{r,1}E_{ij}$ gives rise
to a superalgebra epimorphism
\begin{equation*}
  \ev\colon Y^\rtt(\gl(V))\twoheadrightarrow U(\gl(V)).
\end{equation*}
\end{Lem}

\begin{proof}
Straightforward.
\end{proof}

The superalgebra $Y^\rtt(\gl(V))$ is also endowed with two different filtrations,
defined via
\begin{equation}\label{gradings}
  \deg_1(t^{(r)}_{ij})=r\ \ \mathrm{and}\ \ \deg_2(t^{(r)}_{ij})=r-1.
\end{equation}
Let $\gr_1 Y^\rtt(\gl(V)), \gr_2 Y^\rtt(\gl(V))$ denote the corresponding
associated graded superalgebras.

\begin{Lem}\label{associated graded}
(a) The assignment $t^{(r)}_{ij}\mapsto \mathsf{t}^{(r)}_{ij}$ gives rise to
a superalgebra isomorphism
\begin{equation}\label{quasiclassical limit 2}
  \gr_1 Y^\rtt(\gl(V))\iso \BC[\{\mathsf{t}^{(r)}_{ij}\}_{1\leq i,j\leq n}^{r\geq 1}]
\end{equation}
with the polynomial superalgebra in the variables $\mathsf{t}^{(r)}_{ij}$ with
the $\BZ_2$-grading $|\mathsf{t}^{(r)}_{ij}|=\ol{i}+\ol{j}$.

\noindent
(b) The assignment $t^{(r)}_{ij}\mapsto (-1)^{\ol{i}}E_{ij}\cdot t^{r-1}$
gives rise to a superalgebra isomorphism
\begin{equation}\label{quasiclassical limit 1}
  \gr_2 Y^\rtt(\gl(V))\iso U(\gl(V)[t])
\end{equation}
with the universal enveloping of the loop superalgebra $\gl(V)[t]=\gl(V)\otimes \BC[t]$.
\end{Lem}

\begin{proof}
Analogous to~\cite[Theorem 1, Corollary 1]{g2},
cf.~\cite[Proposition 2.2, Corollary 2.3]{p2}.
\end{proof}

Let us now relate $Y^\rtt(\gl(V))$ to $Y(\gl(V))$.
Consider the Gauss decomposition of $T(z)$:
\begin{equation*}
  T(z)=F(z)\cdot D(z)\cdot E(z).
\end{equation*}
Here,
  $F(z),D(z),E(z)\in \left(Y^\rtt(\gl(V))\otimes \End\ V\right)[[z^{-1}]]$
are of the form
\begin{equation*}
\begin{split}
  & F(z)=\sum_{i} E_{ii}+\sum_{j<i} (-1)^{\ol{j}(\ol{i}+1)}F_{ij}(z)\otimes E_{ij},\
    D(z)=\sum_{i} D_i(z)\otimes E_{ii},\\
  & E(z)=\sum_{i} E_{ii}+\sum_{j<i} (-1)^{\ol{i}(\ol{j}+1)}E_{ji}(z)\otimes E_{ji},
\end{split}
\end{equation*}
cf.~Remark~\ref{sign explanation}. Define the elements
  $\{D^{(s)}_k,\wt{D}^{(s)}_k,E^{(r)}_{ji},F^{(r)}_{ij}\}_{1\leq j<i\leq n, 1\leq k\leq n}^{r\geq 1, s\geq 0}$
of $Y^\rtt(\gl(V))$ via
\begin{equation*}
  E_{ji}(z)=\sum_{r\geq 1} E^{(r)}_{ji} z^{-r},\
  F_{ij}(z)=\sum_{r\geq 1} F^{(r)}_{ij} z^{-r},\
  D_{k}(z)=\sum_{s\geq 0} D^{(s)}_k z^{-s},\
  D_{k}(z)^{-1}=\sum_{s\geq 0} \wt{D}^{(s)}_k z^{-s}.
\end{equation*}
For $1\leq i<n$ and $r\geq 1$, set
$E^{(r)}_i:=E^{(r)}_{i,i+1}$ and $F^{(r)}_i:=F^{(r)}_{i+1,i}$.
%Similar to~\cite[(5.5)]{bk}, cf.~\cite[(10)]{g2} and~\cite[Lemma 3.3]{p2}
Due to~\cite[Lemma 3.3]{p2} (generalizing~\cite[(5.5)]{bk} in the classical
setup as well as~\cite[(10)]{g2} for the distinguished Dynkin diagram), we have:

\begin{Lem}\label{higher root modes}
For any $1\leq j < i-1 < n$, the following equalities hold in $Y^\rtt(\gl(V))$:
\begin{equation*}
  E^{(r)}_{ji}=(-1)^{\ol{i-1}}[E^{(r)}_{j,i-1},E^{(1)}_{i-1}],\
  F^{(r)}_{ij}=(-1)^{\ol{i-1}}[F^{(1)}_{i-1},F^{(r)}_{i-1,j}].
\end{equation*}
\end{Lem}

%\begin{proof}
%The proof is completely analogous to that of~\cite[Proposition 3.44]{ft2}
%and proceeds by comparing the matrix coefficients
%  $\langle \sfv_j\otimes \sfv_{i-1}|\cdots|\sfv_{i-1}\otimes \sfv_i \rangle$
%and
%  $\langle \sfv_i\otimes \sfv_{i-1}|\cdots|\sfv_{i-1}\otimes \sfv_j \rangle$
%in both sides of the equality~(\ref{RTT matrix form})
%(and using a double induction: first in $j$ and then in $i$).
%\end{proof}

\begin{Cor}\label{surjectivity of upsilon}
$Y^\rtt(\gl(V))$ is generated by
  $\{D_j^{(s)}, \wt{D}_j^{(s)}, E_i^{(r)}, F_i^{(r)}\}_{1\leq i<n, 1\leq j\leq n}^{r\geq 1,s\geq 0}$.
\end{Cor}

Similar to~\cite{df,bk,g2,p2}, we have the following result:

\begin{Thm}\label{Yangian Gauss decomposition}
There is a unique superalgebra isomorphism
\begin{equation}\label{upsilon isom}
  \Upsilon\colon Y(\gl(V))\iso Y^\rtt(\gl(V))
\end{equation}
defined by
  $e^{(r)}_i\mapsto E^{(r)}_i, f^{(r)}_i\mapsto F^{(r)}_i,
   d^{(s)}_j\mapsto D^{(s)}_j, \wt{d}^{(s)}_j\mapsto \wt{D}^{(s)}_j$.
\end{Thm}

\begin{proof}
The proof is completely analogous to that in the classical case (when $n_-=0$)
presented in~\cite[\S5]{bk}; see~\cite[Theorem 3]{g2} for the particular case
of~(\ref{distinguished Dynkin}).
\end{proof}

\begin{Rem}\label{relevance of degree 4 Serre }
The presence of the quartic Serre relations~(\ref{Dr Yangian 11},~\ref{Dr Yangian 12})
is solely due to the fact that they also appear among the defining relations of the
Lie superalgebra $\gl(V)$ via the Chevalley generators~\cite{lss}
(see the argument right after~\cite[(59)]{g2}).
\end{Rem}

\begin{Rem}\label{isomorphism of Drinfeld yangians gl}
Theorem~\ref{Yangian Gauss decomposition} together with
Lemma~\ref{isomorphism of RTT yangians} implies
Theorem~\ref{Drinfeld super yangians gl are iso}.
\end{Rem}

\subsection{The RTT super Yangians of $\ssl(V), A(V)$}\label{ssec RTT super Yangian sl}
\

For any formal power series $f(z)\in 1+z^{-1}\BC[[z^{-1}]]$, the assignment
\begin{equation}\label{yangian automorphims}
  \mu_f\colon T(z)\mapsto f(z)T(z)
\end{equation}
defines a superalgebra automorphism $\mu_f$ of $Y^\rtt(\gl(V))$.
Following~\cite{mno}, define the \emph{RTT super Yangian of $\ssl(V)$},
denoted by $Y^\rtt(\ssl(V))$, as the $\BC$-subalgebra of $Y^\rtt(\gl(V))$ via
\begin{equation}\label{RTT sl-super-Yangian}
  Y^\rtt(\ssl(V)):=\{y\in Y^\rtt(\gl(V))|\mu_f(y)=y\ \mathrm{for\ all}\ f\}.
\end{equation}
In the particular case~(\ref{distinguished Dynkin}), this recovers
the super Yangian $Y^\rtt(\ssl_{n_+|n_-})$ of~\cite[\S8]{g2}.

In view of Lemma~\ref{isomorphism of RTT yangians}, we immediately obtain:

\begin{Cor}\label{isomorphism of RTT yangians sl}
The superalgebra $Y^\rtt(\ssl(V))$ depends only on $(n_+,n_-)$, up to
an isomorphism. In particular, $Y^\rtt(\ssl(V))$ is isomorphic to
the super Yangian $Y^\rtt(\ssl_{n_+|n_-})$ of~\cite{g2}.
\end{Cor}

Explicitly, this subalgebra can be described as follows:

\begin{Lem}\label{explicit RTT sl}
$Y^\rtt(\ssl(V))$ is generated by coefficients of
$D_i(z)^{-1}D_{i+1}(z), E_{i,i+1}(z), F_{i+1,i}(z)$.
\end{Lem}

\begin{proof}
Completely analogous to the proof of~\cite[Lemma 7]{g2}
for the particular case of~(\ref{distinguished Dynkin}).
\end{proof}

\begin{Def}\label{charge}
Define the \emph{charge} $c(V)\in \BZ$ of $V$ via
$c(V):=n_+-n_-=\dim(V_{\bar{0}})-\dim(V_{\bar{1}})$.
\end{Def}

If $V$ has a nonzero charge, then $Y^\rtt(\ssl(V))$ also may be realized
as a quotient of $Y^\rtt(\gl(V))$. For the latter construction, let us
first obtain an explicit description of the center $ZY^\rtt(\gl(V))$ of
$Y^\rtt(\gl(V))$. Following~\cite{g1}, define the \emph{quantum Berezinian}
$b(z)\in Y^\rtt(\gl(V))[[z^{-1}]]$ via
\begin{equation}\label{Berezinian}
  b(z):=1+\sum_{r\geq 1} b_rz^{-r}=D'_1(z_1)D'_2(z_2)\cdots D'_n(z_n),
\end{equation}
where
  $D'_i(z):=\begin{cases}
     D_i(z), & \text{if } \ol{i}=\bar{0}\\
     D_i(z)^{-1}, & \text{if } \ol{i}=\bar{1}
   \end{cases}$,
while $z_1=z$ and
  $z_{i+1}=\begin{cases}
     z_i+c_{i,i+1}, & \text{if } |\alpha_i|=\bar{0}\\
     z_i, & \text{if } |\alpha_i|=\bar{1}
   \end{cases}$.

\begin{Rem}
For the distinguished Dynkin diagram, that is for~(\ref{distinguished Dynkin}),
this definition recovers the original \emph{quantum Berezinian} of~\cite[\S2]{na},
due to the main result (Theorem 1) of~\cite{g1}.
\end{Rem}

\begin{Thm}\label{center of super Yangian}
(a) The elements $\{b_r\}_{r\geq 1}$ are central.

\noindent
(b) The elements $\{b_r\}_{r\geq 1}$ are algebraically independent,
and generate the center $ZY^\rtt(\gl(V))$. In other words, we have
an algebra isomorphism $ZY^\rtt(\gl(V))\simeq \BC[b_1,b_2,\ldots]$.
\end{Thm}

\begin{proof}
(a) To prove that all $b_r$ are central, it suffices to verify
that $[b(z),E_i(w)]=0=[b(z),F_i(w)]$ for any $1\leq i<n$. We shall
check only the first equality (the second is analogous).

\emph{Case 1: $|\alpha_i|=\bar{0}$}.
Due to the isomorphism of Theorem~\ref{Yangian Gauss decomposition}
and the relation~(\ref{Dr Yangian 2}), we have
\begin{equation}\label{auxillary 1}
  (u-v)E_i(v)D_i(u)=(u-v-(-1)^{\ol{i}})D_i(u)E_i(v)+(-1)^{\ol{i}}D_i(u)E_i(u),
\end{equation}
\begin{equation}\label{auxillary 2}
  (w-v)E_i(v)D_{i+1}(w)=
  (w-v+(-1)^{\ol{i+1}})D_{i+1}(w)E_i(v)-(-1)^{\ol{i+1}}D_{i+1}(w)E_i(w).
\end{equation}
Plugging $v=u,w=u-(-1)^{\ol{i}}$ into~(\ref{auxillary 2}) and using
$\ol{i}=\ol{i+1}$ (as $|\alpha_i|=\bar{0}$), we get
\begin{equation}\label{auxillary 3}
  E_i(u)D_{i+1}(u-(-1)^{\ol{i}})=D_{i+1}(u-(-1)^{\ol{i}})E_i(u-(-1)^{\ol{i}}).
\end{equation}
Due to~(\ref{auxillary 1}--\ref{auxillary 3}):
  $(u-v)E_i(v)D_i(u)D_{i+1}(u-(-1)^{\ol{i}})=(u-v)D_i(u)D_{i+1}(u-(-1)^{\ol{i}})E_i(v)$.
Hence, $[b(z),E_i(w)]=0$ as $c_{i,i+1}=-(-1)^{\ol{i+1}}=-(-1)^{\ol{i}}$ and
$[E_i(v),D_j(u)]=0$ for $j\ne i,i+1$.

\emph{Case 2: $|\alpha_i|=\bar{1}$}.
In this case, $(-1)^{\ol{i+1}}=-(-1)^{\ol{i}}$ and the equality~(\ref{auxillary 2}) is equivalent to
\begin{equation}\label{auxillary 4}
  (w-v)D_{i+1}(w)^{-1}E_i(v)=
  (w-v-(-1)^{\ol{i}})E_i(v)D_{i+1}(w)^{-1}+(-1)^{\ol{i}}E_i(w)D_{i+1}(w)^{-1}.
\end{equation}
Combining the equalities~(\ref{auxillary 1},~\ref{auxillary 4}), we immediately obtain
  $(u-v)E_i(v)D_i(u)D_{i+1}(u)^{-1}=(u-v)D_i(u)D_{i+1}(u)^{-1}E_i(v)$.
Hence, $[b(z),E_i(w)]=0$ as $[E_i(v),D_j(u)]=0$ for $j\ne i,i+1$.

This completes the proof of part (a).

\noindent
(b) The proof of part (b) is analogous to that of~\cite[Theorem 2.13]{mno}
and~\cite[Theorem 7.2]{bk} in the classical case (when $n_-=0$), and
of~\cite[Theorem 4]{g2} for the particular case~(\ref{distinguished Dynkin}).
\end{proof}

Similar to the classical case (when $n_-=0$) treated in~\cite{mno} as well as the
particular case~(\ref{distinguished Dynkin}) treated in~\cite{g2}, we have:

\begin{Thm}\label{gl vs sl and center}
(a) If $c(V)\ne 0$, then we have a superalgebra isomorphism
\begin{equation}\label{RTT gln vs sln}
  Y^\rtt(\gl(V))\simeq Y^\rtt(\ssl(V))\otimes ZY^\rtt(\gl(V)).
\end{equation}

\noindent
(b) If $c(V)=0$, then $ZY^\rtt(\gl(V))\subset Y^\rtt(\ssl(V))$.
\end{Thm}

\begin{proof}
(a) Analogous to the proof of~\cite[Proposition 3]{g2}
for the particular case of~(\ref{distinguished Dynkin}).

\noindent
(b) If $n_+=n_-$, then $z_i$ defined after~(\ref{Berezinian}) satisfy
$\{z_i|\ol{i}=\bar{0}\}=\{z_i|\ol{i}=\bar{1}\}$. Hence, $\mu_f(b(z))=b(z)$
for all automorphisms~(\ref{yangian automorphims}). Thus,
$ZY^\rtt(\gl(V))\subset Y^\rtt(\ssl(V))$ by Theorem~\ref{center of super Yangian}.
\end{proof}

\begin{Cor}\label{sl as a factor of gl}
If $c(V)\ne 0$, then the isomorphism~(\ref{RTT gln vs sln}) gives rise
to a natural epimorphism
  $\pi\colon Y^\rtt(\gl(V))\twoheadrightarrow Y^\rtt(\ssl(V))$
with $\Ker(\pi)=(b_1,b_2,\ldots)$.
\end{Cor}

Recall that the classical Lie superalgebra $A(n_+ - 1, n_- - 1)$ coincides
with $\ssl(n_+|n_-)$ for $n_+\ne n_-$, and with the quotient $\ssl(n_+|n_-)/(I)$
for $n_+=n_-$, where $I=\sum_{i=1}^n E_{ii}$ is the central element.
Motivated by this and Theorem~\ref{gl vs sl and center}(b), if $c(V)=0$,
define the \emph{RTT super Yangian of $A(V)$}, denoted by $Y^\rtt(A(V))$,
via $Y^\rtt(A(V)):=Y^\rtt(\ssl(V))/(b_1,b_2,\ldots)$, cf.~\cite[(67)]{g2}.

\begin{Cor}\label{isomorphism of RTT yangians A-type}
$Y^\rtt(A(V))$ depends only on $n_+ = n_-$, up to an isomorphism.
\end{Cor}

\begin{proof}
Similar to~\cite[Corollary 2]{g2}, the center $ZY^\rtt(\ssl(V))$ of
$Y^\rtt(\ssl(V))$ is a polynomial algebra in $\{b_r\}_{r=1}^\infty$.
Combining this with Corollary~\ref{isomorphism of RTT yangians sl}
implies the result.
\end{proof}

    %%%%%%%%%%%%%%%%%%%%%%%%%%%%%%%%%%%%%%%%%%%%%%%%%%%%%%%%%%%%%%%%%%%%%%%%%%%%%%%
    %%%%%%%%%%%%%%%%%%%%%%%%%%%%%%% Drinfeld Super Yangian sl %%%%%%%%%%%%%%%%%%%%%
    %%%%%%%%%%%%%%%%%%%%%%%%%%%%%%%%%%%%%%%%%%%%%%%%%%%%%%%%%%%%%%%%%%%%%%%%%%%%%%%

\subsection{The Drinfeld super Yangian of $\ssl(V)$}\label{ssec super Yangian sl}
\

Following~\cite{d} (cf.~\cite{s}\footnote{As noticed in~\cite{g2},
the relation~(\ref{Dr Yangian sl 6 generalized}) should replace the
wrong quartic Serre relations of~\cite[Definition~2]{s}.} and~\cite{g2}),
define the \emph{Drinfeld super Yangian of $\ssl(V)$}, denoted by $Y(\ssl(V))$,
to be the associative $\BC$-superalgebra generated by
$\{h_{i,r}, \sx^\pm_{i,r}\}_{1\leq i<n}^{r\geq 0}$ with the
$\BZ_2$-grading $|h_{i,r}|=\bar{0}$, $|\sx^\pm_{i,r}|=|\alpha_i|$,
and subject to the following defining relations:
\begin{equation}\label{Dr Yangian sl 1}
  [h_{i,r}, h_{j,s}]=0,
\end{equation}
\begin{equation}\label{Dr Yangian sl 2.1}
  [h_{i,0},\sx^\pm_{j,s}]=\pm c_{ij} \sx^\pm_{j,s},
\end{equation}
\begin{equation}\label{Dr Yangian sl 2.2}
  [h_{i,r+1},\sx^\pm_{j,s}]-[h_{i,r},\sx^\pm_{j,s+1}]=
  \pm\frac{c_{ij}}{2}\{h_{i,r},\sx^\pm_{j,s}\}\
  \mathrm{unless}\ i=j\ \mathrm{and}\ |\alpha_i|=\bar{1},
\end{equation}
\begin{equation}\label{Dr Yangian sl 2.3}
  [h_{i,r},\sx^\pm_{i,s}]=0 \ \mathrm{if}\ |\alpha_i|=\bar{1},
\end{equation}
\begin{equation}\label{Dr Yangian sl 3}
  [\sx^+_{i,r},\sx^-_{j,s}]=\delta_{i,j}h_{i,r+s},
\end{equation}
\begin{equation}\label{Dr Yangian sl 4.1}
  [\sx^\pm_{i,r+1},\sx^\pm_{j,s}]-[\sx^\pm_{i,r},\sx^\pm_{j,s+1}]=
  \pm\frac{c_{ij}}{2}\{\sx^\pm_{i,r},\sx^\pm_{j,s}\}\
  \mathrm{unless}\ i=j\ \mathrm{and}\ |\alpha_i|=\bar{1},
\end{equation}
\begin{equation}\label{Dr Yangian sl 4.2}
  [\sx^\pm_{i,r},\sx^\pm_{j,s}]=0 \ \mathrm{if}\ c_{ij}=0,
\end{equation}
as well as cubic Serre relations
\begin{equation}\label{Dr Yangian sl 5}
  [\sx^\pm_{i,r},[\sx^\pm_{i,s},\sx^\pm_{j,t}]]+
  [\sx^\pm_{i,s},[\sx^\pm_{i,r},\sx^\pm_{j,t}]]=0
  \ \mathrm{if}\ j=i\pm 1\ \mathrm{and}\ |\alpha_i|=\bar{0},
\end{equation}
and quartic Serre relations
\begin{equation}\label{Dr Yangian sl 6}
  [[\sx^\pm_{j-1,r},\sx^\pm_{j,0}],[\sx^\pm_{j,0},\sx^\pm_{j+1,s}]]=0
  \ \mathrm{if}\ |\alpha_j|=\bar{1}\ \mathrm{and}\ |\alpha_{j-1}|=|\alpha_{j+1}|=\bar{0}.
\end{equation}

\begin{Rem}\label{Serre hold always 2}
(a) Similar to Remark~\ref{Serre hold always}, Serre
relations~(\ref{Dr Yangian sl 5}) and~(\ref{Dr Yangian sl 6}) also
hold for all other parities, but in those cases, they already follow
from~(\ref{Dr Yangian sl 4.1},~\ref{Dr Yangian sl 4.2})
and~(\ref{Dr Yangian sl 4.1},~\ref{Dr Yangian sl 4.2},~\ref{Dr Yangian sl 5}).

\noindent
(b) Generalizing the quartic Serre relations~(\ref{Dr Yangian sl 6}),
the following relations also hold:
\begin{equation}\label{Dr Yangian sl 6 generalized}
  [[\sx^\pm_{j-1,r},\sx^\pm_{j,k}],[\sx^\pm_{j,l},\sx^\pm_{j+1,s}]]+
  [[\sx^\pm_{j-1,r},\sx^\pm_{j,l}],[\sx^\pm_{j,k},\sx^\pm_{j+1,s}]]=0.
\end{equation}
One way to prove this is to use the classical argument of deducing all Serre
relations from the basic ones by commuting the latter with certain Cartan elements.
Let $Q^\pm_j(r;k,l;s)$ denote the left-hand side of~(\ref{Dr Yangian sl 6 generalized}).
Our goal is to prove $Q^\pm_j(r;k,l;s)=0$ for any $r,k,l,s\geq 0$, while we know it
only for $k=l=0$ and $r,s\geq 0$, due to~(\ref{Dr Yangian sl 6}) and Remark~\ref{Serre hold always 2}(a).
Define the elements $\{t_{i,r}\}_{1\leq i<n}^{r\geq 0}$ of $Y(\ssl(V))$ via
\begin{equation}\label{t-elements}
  \sum_{r\geq 0} t_{i,r}u^{-r-1}=\log\left(1+\sum_{r\geq 0} h_{i,r}u^{-r-1}\right),
\end{equation}
cf.~(\ref{Dr Yangian sl 1}). The relations~(\ref{Dr Yangian sl 2.1},~\ref{Dr Yangian sl 2.2})
imply the following commutation relations:
\begin{equation}\label{shift operator}
  [t_{i,r},\sx^\pm_{j,s}]=\pm c_{ij}
  \sum_{l=0}^{[r/2]} \binom{r}{2l} \frac{(c_{ij}/2)^{2l}}{2l+1}\sx^\pm_{j,r+s-2l},
\end{equation}
cf.~\cite[Remark of \S2.9]{gtl}.
Commuting both sides of the equality $Q^\pm_j(r;0,0;s)=0$ with $t_{j-1,k}$
and using~(\ref{shift operator}), one obtains $Q^\pm_j(r;k,0;s)=0$ by
an induction in $k$. Commuting the latter equality with $t_{j-1,l}$,
one derives the desired equality $Q^\pm_j(r;k,l;s)=0$ by an induction in $l$.
Another way to prove~(\ref{Dr Yangian sl 6 generalized}) is to verify
the corresponding equality on the shuffle side, see
Section~\ref{sec shuffle for superYangians} (though we set $\hbar=1$ for the current purpose).
According to Corollary~\ref{Injectivity of Psi}, it suffices to prove
$\Psi(Q^+_j(r;k,l;s))=0$. The latter follows from the obvious equality
\begin{equation}\label{elegant deduction of higher Serre}
   \Psi(Q^+_j(r;k,l;s))=
   (x_{j,1}^{k}x_{j,2}^{l}+x_{j,1}^{l}x_{j,2}^{k})\Psi(Q^+_j(r;0,0;s))
\end{equation}
in the notations of \emph{loc.cit.}, combined with $Q^+_j(r;0,0;s)=0$.
\end{Rem}

Let us now relate $Y(\ssl(V))$ to $Y(\gl(V))$ of Section~\ref{ssec Drinfeld super Yangian gl}.
Define $u_1,\ldots,u_{n-1}$ via
\begin{equation}\label{u-shifts}
  u_1:=u\ \mathrm{and}\
  u_{i+1}=u_i+\frac{c_{i,i+1}}{2}=u_i-\frac{(-1)^{\ol{i+1}}}{2}.
\end{equation}
Consider the generating series $e_i(u),f_i(u),d_j(u)$ with
coefficients in $Y(\gl(V))$, defined via
\begin{equation*}
  e_i(u):=\sum_{r\geq 1} e^{(r)}_iu^{-r},\
  f_i(u):=\sum_{r\geq 1} f^{(r)}_iu^{-r},\
  d_j(u):=1+\sum_{r\geq 1} d^{(r)}_ju^{-r}.
\end{equation*}
We also introduce the elements $\{X^\pm_{i,r}, H_{i,r}\}_{1\leq i<n}^{r\geq 0}$
of $Y(\gl(V))$ via
\begin{equation*}
  \sum_{r\geq 0} X^{+}_{i,r}u^{-r-1}=f_i(u_i),\
  \sum_{r\geq 0} X^{-}_{i,r}u^{-r-1}=(-1)^{\ol{i}}e_i(u_i),\
  1+\sum_{r\geq 0} H_{i,r}u^{-r-1}=d_i(u_i)^{-1}d_{i+1}(u_i).
\end{equation*}

\begin{Thm}\label{sl vs gl Drinfeld}
The assignment
  $\sx^\pm_{i,r}\mapsto X^\pm_{i,r}, h_{i,r}\mapsto H_{i,r}$
gives rise to a superalgebra embedding
\begin{equation}\label{sl embedded into gl}
  \jmath\colon Y(\ssl(V))\hookrightarrow Y(\gl(V)).
\end{equation}
Moreover, the superalgebra isomorphism
  $\Upsilon\colon Y(\gl(V))\iso Y^\rtt(\gl(V))$
of Theorem~\ref{Yangian Gauss decomposition} gives rise to a superalgebra
isomorphism $\Upsilon\colon Y(\ssl(V))\iso Y^\rtt(\ssl(V))$.
\end{Thm}

\begin{proof}
The compatibility of the assignment
  $\sx^\pm_{i,r}\mapsto X^\pm_{i,r}, h_{i,r}\mapsto H_{i,r}$
with the defining relations~(\ref{Dr Yangian sl 1}--\ref{Dr Yangian sl 6}) is
straightforward (cf.~\cite[Remark 5.12]{bk}). Hence, we obtain a superalgebra
homomorphism $\jmath\colon Y(\ssl(V))\to Y(\gl(V))$. Its image coincides with
the pre-image of $Y^\rtt(\ssl(V))$ under $\Upsilon$, due to Lemma~\ref{explicit RTT sl}.
Finally, the injectivity of $\jmath$ is established in the same way as it was
proved in~\cite[Proposition 5]{g2} for the particular case of $Y(\gl_{n_+|n_-})$
(the proof crucially uses the construction of PBW bases for the Yangians
of~\cite{l,s}, recalled in Theorem~\ref{PBW for Yangian}).
\end{proof}

As an immediate corollary of Theorem~\ref{sl vs gl Drinfeld} and
Corollary~\ref{isomorphism of RTT yangians sl}, we obtain:

\begin{Thm}\label{isomorphism of Drinfeld yangians sl}
The superalgebra $Y(\ssl(V))$ depends only on $(n_+,n_-)$, up to an isomorphism.
\end{Thm}

    %%%%%%%%%%%%%%%%%%%%%%%%%%%%%%%%%%%%%%%%%%%%%%%%%%%%%%%%%%%%%%%%%%%%%%%%%%%%%%%
    %%%%%%%%%%%%%%%%%%%%% PBW Theorems and Trinagular Decomposition %%%%%%%%%%%%%%%
    %%%%%%%%%%%%%%%%%%%%%%%%%%%%%%%%%%%%%%%%%%%%%%%%%%%%%%%%%%%%%%%%%%%%%%%%%%%%%%%

\subsection{The PBW theorem and the triangular decomposition for $Y(\ssl(V))$}\label{ssec triangular decomposition}
\

Let $Y^{\pm}(\ssl(V))$ and $Y^{0}(\ssl(V))$ be the subalgebras of $Y(\ssl(V))$
generated by $\{\sx^\pm_{i,r}\}$ and $\{h_{i,r}\}$, respectively. Likewise, let
$\wt{Y}^{\pm}(\ssl(V))$ and $\wt{Y}^{0}(\ssl(V))$ be the associative
$\BC$-superalgebras generated by $\{\sx^\pm_{i,r}\}_{1\leq i<n}^{r\geq 0}$
and $\{h_{i,r}\}_{1\leq i<n}^{r\geq 0}$, respectively, with the $\BZ_2$-grading
$|h_{i,r}|=\bar{0}, |\sx^\pm_{i,r}|=|\alpha_i|$, and subject to the defining
relations~(\ref{Dr Yangian sl 4.1}--\ref{Dr Yangian sl 6})
and~(\ref{Dr Yangian sl 1}), respectively. The assignments
$\sx^\pm_{i,r}\mapsto \sx^\pm_{i,r}$ and $h_{i,r}\mapsto h_{i,r}$
clearly give rise to epimorphisms
  $\wt{Y}^{\pm}(\ssl(V))\twoheadrightarrow Y^{\pm}(\ssl(V))$ and
  $\wt{Y}^{0}(\ssl(V))\twoheadrightarrow Y^{0}(\ssl(V))$
(which are actually isomorphisms, due to
Proposition~\ref{Triangular decomposition}(a)).

Pick any total ordering $\preceq$ on $\Delta^+\times \BN$.
For every $(\beta,r)\in \Delta^+\times \BN$, we choose:

(1) a decomposition $\beta=\alpha_{i_1}+\ldots+\alpha_{i_p}$ such that
$[\cdots[e_{\alpha_{i_1}},e_{\alpha_{i_2}}],\cdots,e_{\alpha_{i_p}}]$
is a nonzero root vector $e_\beta$ of $\ssl(V)$
(here, $e_{\alpha_i}$ denotes the standard Chevalley generator of $\ssl(V)$);

(2) a decomposition $r=r_1+\ldots+r_p$ with $r_i\in \BN$.

\noindent
Define the \emph{PBW basis elements} $\sx^\pm_{\beta,r}$ of
$Y^{\pm}(\ssl(V))$ or $\wt{Y}^{\pm}(\ssl(V))$ via
\begin{equation}\label{higher roots}
  \sx^\pm_{\beta,r}:=
  [\cdots[[\sx^\pm_{i_1,r_1},\sx^\pm_{i_2,r_2}],\sx^\pm_{i_3,r_3}],\cdots,\sx^\pm_{i_p,r_p}].
\end{equation}

Let $H$ denote the set of all functions $h\colon \Delta^+\times \BN\to \BN$
with finite support and such that $h(\beta,r)\leq 1$ if $|\beta|=\bar{1}$
(we set $|\pm(\alpha_j+\ldots+\alpha_i)|:=|\alpha_j|+\ldots+|\alpha_i|\in \BZ_2$).
The monomials
\begin{equation}\label{PBWD monomials}
  \sx^\pm_h:=
  \prod\limits_{(\beta,r)\in \Delta^+\times \BN}^{\rightarrow} {\sx^\pm_{\beta,r}}^{h(\beta,r)}
  \ \ \mathrm{with}\ h\in H
\end{equation}
will be called the \emph{ordered PBW monomials} of $Y^{\pm}(\ssl(V))$
or $\wt{Y}^{\pm}(\ssl(V))$.

The following PBW result for Yangians is originally due to~\cite{l}\footnote{The
original proof of~\cite{l} contains a substantial gap, see~\cite[Appendix B]{ft2}
for an alternative proof.} (cf.~\cite[Theorem 5.11]{bk} and \cite[proof of Proposition 5]{g2}):

\begin{Thm}\label{PBW for Yangian}
(a) The ordered PBW monomials $\{\sx^\pm_h\}_{h\in H}$ form
a $\BC$-basis of $\wt{Y}^{\pm}(\ssl(V))$.

\noindent
(b) The ordered (in any way) monomials in $\{h_{i,r}\}_{1\leq i<n}^{r\geq 0}$
form a $\BC$-basis of $\wt{Y}^{0}(\ssl(V))$.

\noindent
(c) The products of ordered PBW monomials
$\{\sx^-_h\}_{h\in H}$, $\{\sx^+_{h'}\}_{h'\in H}$, and the ordered monomials
in $\{h_{i,r}\}_{1\leq i<n}^{r\geq 0}$ form a $\BC$-basis of $Y(\ssl(V))$.
\end{Thm}

As an important corollary, we obtain the \emph{triangular decomposition} for $Y(\ssl(V))$:

\begin{Prop}\label{Triangular decomposition}
(a) The assignments $\sx^\pm_{i,r}\mapsto \sx^\pm_{i,r}$ and
$h_{i,r}\mapsto h_{i,r}$ give rise to isomorphisms
\begin{equation*}
  \wt{Y}^{\pm}(\ssl(V))\iso Y^{\pm}(\ssl(V))
  \ \mathrm{and}\
  \wt{Y}^{0}(\ssl(V))\iso Y^{0}(\ssl(V)).
\end{equation*}

\noindent
(b) The multiplication map
\begin{equation*}
  m\colon Y^{-}(\ssl(V))\otimes Y^{0}(\ssl(V))\otimes Y^{+}(\ssl(V))
  \longrightarrow Y(\ssl(V))
\end{equation*}
is an isomorphism of $\BC$-vector superspaces.
\end{Prop}

\begin{Rem}\label{special order}
In Section~\ref{sec shuffle for superYangians}, we will use a particular
total ordering $\preceq$ on $\Delta^+\times \BN$:
\begin{equation}\label{double order}
  (\beta,r)\preceq (\beta',r')\ \mathrm{iff}\ \beta\prec \beta'
  \ \mathrm{or}\ \beta=\beta', r\leq r',
\end{equation}
where the total ordering $\preceq$ on $\Delta^+$ is as follows:
\begin{equation}\label{order}
  \alpha_j+\alpha_{j+1}+\ldots+\alpha_i\preceq \alpha_{j'}+\alpha_{j'+1}+\ldots+\alpha_{i'}
  \ \mathrm{iff}\ j<j'\ \mathrm{or}\ j=j',i\leq i'.
\end{equation}
\end{Rem}

    %%%%%%%%%%%%%%%%%%%%%%%%%%%%%%%%%%%%%%%%%%%%%%%%%%%%%%%%%%%%%%%%%%%%%%%%%%%%%%%
    %%%%%%%%%%%%%%%%%%%%%%%%%% Formal version of the Yangian %%%%%%%%%%%%%%%%%%%%%%
    %%%%%%%%%%%%%%%%%%%%%%%%%%%%%%%%%%%%%%%%%%%%%%%%%%%%%%%%%%%%%%%%%%%%%%%%%%%%%%%

\subsection{The super Yangians $Y_\hbar(\ssl(V))$ and $\bY_\hbar(\ssl(V))$}\label{ssec formal Yangian}
\

For the sake of the next section, let us introduce a $\BC[\hbar]$-version
of $Y(\ssl(V))$ by homogenizing the defining relations of the latter.
More precisely, let $Y_\hbar(\ssl(V))$ be the associative $\BC[\hbar]$-superalgebra
generated by $\{h_{i,r}, \sx^\pm_{i,r}\}_{1\leq i<n}^{r\geq 0}$
with the $\BZ_2$-grading $|h_{i,r}|=\bar{0}, |\sx^\pm_{i,r}|=|\alpha_i|$, and
subject to~(\ref{Dr Yangian sl 1},~\ref{Dr Yangian sl 2.1},~\ref{Dr Yangian sl 2.3},~\ref{Dr Yangian sl 3},~\ref{Dr Yangian sl 4.2},~\ref{Dr Yangian sl 5},~\ref{Dr Yangian sl 6})
and the following modifications of~(\ref{Dr Yangian sl 2.2},~\ref{Dr Yangian sl 4.1}):
\begin{equation}\label{Dr Yangian sl 2 formal}
  [h_{i,r+1},\sx^\pm_{j,s}]-[h_{i,r},\sx^\pm_{j,s+1}]=
  \pm\frac{c_{ij}\hbar}{2}\{h_{i,r},\sx^\pm_{j,s}\}\
  \mathrm{unless}\ i=j\ \mathrm{and}\ |\alpha_i|=\bar{1},
\end{equation}
\begin{equation}\label{Dr Yangian sl 4 formal}
  [\sx^\pm_{i,r+1},\sx^\pm_{j,s}]-[\sx^\pm_{i,r},\sx^\pm_{j,s+1}]=
  \pm\frac{c_{ij}\hbar}{2}\{\sx^\pm_{i,r},\sx^\pm_{j,s}\}\
  \mathrm{unless}\ i=j\ \mathrm{and}\ |\alpha_i|=\bar{1}.
\end{equation}
The algebra $Y_\hbar(\ssl(V))$ is $\BN$-graded via
  $\deg(h_{i,r})=\deg(\sx^\pm_{i,r})=r, \deg(\hbar)=1$.

Following Section~\ref{ssec triangular decomposition}, let
$Y^\pm_\hbar(\ssl(V))$ and $Y^0_\hbar(\ssl(V))$ be the
$\BC[\hbar]$-subalgebras of $Y_\hbar(\ssl(V))$ generated by
$\{\sx^\pm_{i,r}\}$ and $\{h_{i,r}\}$, respectively.
We also define the \emph{PBW basis elements}
$\{\sx^\pm_{\beta,r}\}_{\beta\in \Delta^+}^{r\in \BN}$ and
the \emph{ordered PBW monomials} $\{\sx^\pm_h\}_{h\in H}$ of $Y_\hbar(\ssl(V))$
via~(\ref{higher roots}) and~(\ref{PBWD monomials}), respectively.
We have the following counterparts of Theorem~\ref{PBW for Yangian}(c)
and Proposition~\ref{Triangular decomposition} (cf.~\cite{ft2}):

\begin{Thm}\label{PBW for formal Yangian}
(a) The products of ordered PBW monomials $\{\sx^-_h\}_{h\in H}$,
$\{\sx^+_{h'}\}_{h'\in H}$, and the ordered monomials in
$\{h_{i,r}\}_{1\leq i<n}^{r\geq 0}$ form a basis of a free
$\BC[\hbar]$-module $Y_\hbar(\ssl(V))$.

\noindent
(b) The multiplication map
  $m\colon
   Y^{-}_\hbar(\ssl(V))\otimes_{\BC[\hbar]} Y^{0}_\hbar(\ssl(V))\otimes_{\BC[\hbar]} Y^{+}_\hbar(\ssl(V))
   \to Y_\hbar(\ssl(V))$
is an isomorphism of $\BC[\hbar]$-modules.

\noindent
(c) $Y^\pm_\hbar(\ssl(V))$ are isomorphic to the associative
$\BC[\hbar]$-superalgebras generated by $\{\sx^\pm_{i,r}\}_{1\leq i<n}^{r\geq 0}$
with the $\BZ_2$-grading $|\sx^\pm_{i,r}|=|\alpha_i|$ and subject to the defining
relations~(\ref{Dr Yangian sl 4.2}--\ref{Dr Yangian sl 6},~\ref{Dr Yangian sl 4 formal}).
\end{Thm}

The Drinfeld-Gavarini dual $\bY_\hbar(\ssl(V))$ is the
$\BC[\hbar]$-subalgebra of $Y_\hbar(\ssl(V))$ generated by
\begin{equation}\label{normalized generators}
  \mathsf{H}_{i,r}:=\hbar\cdot h_{i,r}\
  \mathrm{and}\
  \SX^\pm_{\beta,r}:=\hbar \sx^\pm_{\beta,r}
  \ \mathrm{for}\ i\in I,\beta\in \Delta^+,r\in \BN.
\end{equation}
For $h\in H$ (Section~\ref{ssec triangular decomposition}), set
  $\SX^\pm_h:=
   \prod\limits_{(\beta,r)\in \Delta^+\times \BN}^{\rightarrow} {\SX^\pm_{\beta,r}}^{h(\beta,r)}$.
The following is~\cite[Theorem A.7]{ft2}:

\begin{Thm}\label{pbw for entire gavarini yangian}
(a) The subalgebra $\bY_\hbar(\ssl(V))$ is independent
of all our choices in~(\ref{higher roots}).

\noindent
(b) The products of ordered PBW monomials $\{\SX^-_h\}_{h\in H}$,
$\{\SX^+_{h'}\}_{h'\in H}$, and the ordered monomials in
$\{\mathsf{H}_{i,r}\}_{1\leq i<n}^{r\geq 0}$ form a basis
of a free $\BC[\hbar]$-module $\bY_\hbar(\ssl(V))$.
\end{Thm}

Let $\bY^+_\hbar(\ssl(V))$ be the $\BC[\hbar]$-subalgebra of $Y^+_\hbar(\ssl(V))$
generated by $\{\SX^+_{\beta,r}\}_{\beta\in \Delta^+}^{r\in \BN}$.
A new proof of Theorem~\ref{pbw for entire gavarini yangian} but with
$\bY^+_\hbar(\ssl(V))$ in place of $\bY_\hbar(\ssl(V))$ is provided
in the next section.

\begin{Rem}\label{formal vs nonformal}
(a) In view of Theorems~\ref{PBW for formal Yangian}
and~\ref{pbw for entire gavarini yangian}, the algebras $Y_\hbar(\ssl(V))$
and $\bY_\hbar(\ssl(V))$ may be defined  as the Rees algebras:
\begin{equation}\label{Rees algebras}
  Y_\hbar(\ssl(V))=\Rees^{F_\ast} Y(\ssl(V))    
  \ \mathrm{and}\
  \bY_\hbar(\ssl(V))=\Rees^{F'_\ast} Y(\ssl(V)).
\end{equation}
Here, $F'_{\ast} Y(\ssl(V))$ and $F_{\ast} Y(\ssl(V))$ are the two
algebra filtrations on $Y(\ssl(V))$, defined by specifying the degrees
of PBW basis elements
  $\{\sx^\pm_{\beta,r},h_{i,r}\}_{\beta\in \Delta^+, 1\leq i<n}^{r\geq 0}$
as follows:
\begin{equation}\label{Gradings}
  \deg_1(\sx^\pm_{\beta,r})=\deg_1(h_{i,r})=r+1
  \ \ \mathrm{and}\ \
  \deg_2(\sx^\pm_{\beta,r})=\deg_2(h_{i,r})=r.
\end{equation}
They are pre-images of the filtrations~(\ref{gradings}) under the embedding
$\Upsilon\colon Y(\ssl(V))\hookrightarrow Y^\rtt(\gl(V))$.

\noindent
(b) For $a\in \BC^\times$,
  $Y_\hbar(\ssl(V))/(\hbar-a) Y_\hbar(\ssl(V))\simeq
   \bY_\hbar(\ssl(V))/(\hbar-a) \bY_\hbar(\ssl(V))\simeq Y(\ssl(V))$,
but
  $Y_\hbar(\ssl(V))/\hbar Y_\hbar(\ssl(V))\simeq U(\ssl(V)\otimes \BC[t])$,
while $\bY_\hbar(\ssl(V))/\hbar \bY_\hbar(\ssl(V))$ is supercommutative.
\end{Rem}

\section{Shuffle algebra realizations of $Y^+_\hbar(\ssl(V))$ and $\bY^+_\hbar(\ssl(V))$}
\label{sec shuffle for superYangians}

In this section, we obtain shuffle algebra realizations\footnote{These
are rational super counterparts of the elliptic shuffle algebras of
Feigin-Odesskii~\cite{fo1}--\cite{fo3}.} of the superalgebras $Y^+_\hbar(\ssl(V))$
and $\bY^+_\hbar(\ssl(V))$ of Section~\ref{ssec formal Yangian},
generalizing~\cite[Theorems 7.15, 7.16]{t} for the particular case of~(\ref{distinguished Dynkin}).

    %%%%%%%%%%%%%%%%%%%%%%%%%%%%%%%%%%%%%%%%%%%%%%%%%%%%%%%%%%%%%%%%%%%%%%%%%
    %%%%%%%%%%%%% Shuffle SuperAlgebra of type A for any V  %%%%%%%%%%%%%%%%%
    %%%%%%%%%%%%%%%%%%%%%%%%%%%%%%%%%%%%%%%%%%%%%%%%%%%%%%%%%%%%%%%%%%%%%%%%%

\subsection{The rational shuffle algebra $W^V$ and its integral form $\fW^V$}
\label{ssec super shuffle algebra}
\

We follow the notations of~\cite[\S7.2]{t}. Let $\Sigma_k$ denote the symmetric
group in $k$ elements, and set
  $\Sigma_{(k_1,\ldots,k_{n-1})}:=\Sigma_{k_1}\times \cdots\times \Sigma_{k_{n-1}}$
for $k_1,\ldots,k_{n-1}\in \BN$. Consider an $\BN^I$-graded $\BC[\hbar]$-module
  $\bar{\BW}^{V}=
   \underset{\underline{k}=(k_1,\ldots,k_{n-1})\in \BN^{I}}
   \bigoplus \bar{\BW}^{V}_{\underline{k}}$,
where $\bar{\BW}^{V}_{(k_1,\ldots,k_{n-1})}$ consists of rational functions
from $\BC[\hbar](\{x_{i,r}\}_{i\in I}^{1\leq r\leq k_i})$ which are
\emph{supersymmetric} in $\{x_{i,r}\}_{r=1}^{k_i}$ for any $i\in I$, that is,
symmetric if $|\alpha_i|=\bar{0}$ and skew-symmetric if $|\alpha_i|=\bar{1}$.

We fix an $I\times I$ matrix of rational functions
  $(\zeta_{i,j}(z))_{i,j\in I} \in \mathrm{Mat}_{I\times I}(\BC[\hbar](z))$
via
\begin{equation}\label{zeta-factor}
  \zeta_{i,j}(z)=(-1)^{\delta_{i>j}\delta_{|\alpha_i|,\bar{1}}\delta_{|\alpha_j|,\bar{1}}}
                 \left(1+\frac{c_{ij}\hbar}{2z}\right).
\end{equation}
Let us now introduce the bilinear \emph{shuffle product} $\star$ on $\bar{\BW}^{V}$:
given  $F\in \bar{\BW}^{V}_{\underline{k}}$ and $G\in \bar{\BW}^{V}_{\underline{l}}$,
define $F\star G\in \bar{\BW}^{V}_{\underline{k}+\underline{l}}$ via
\begin{equation}\label{shuffle product}
\begin{split}
  & (F\star G)(x_{1,1},\ldots,x_{1,k_1+l_1};\ldots;x_{n-1,1},\ldots, x_{n-1,k_{n-1}+l_{n-1}}):=
    \unl{k}!\cdot\unl{l}!\times\\
  & \SSym_{\Sigma_{\unl{k}+\unl{l}}}
    \left(F\left(\{x_{i,r}\}_{i\in I}^{1\leq r\leq k_i}\right)
          G\left(\{x_{i',r'}\}_{i'\in I}^{k_{i'}<r'\leq k_{i'}+l_{i'}}\right)\cdot
          \prod_{i\in I}^{i'\in I}\prod_{r\leq k_i}^{r'>k_{i'}}\zeta_{i,i'}(x_{i,r}-x_{i',r'})\right).
\end{split}
\end{equation}
Here, $\unl{k}!=\prod_{i\in I}k_i!$, and for
  $f\in \BC(\{x_{i,1},\ldots,x_{i,m_i}\}_{i\in I})$
we define its \emph{supersymmetrization} via
\begin{equation*}
  \SSym_{\Sigma_{\unl{m}}}(f)(\{x_{i,1},\ldots,x_{i,m_i}\}_{i\in I}):=
  \sum_{(\sigma_1,\ldots,\sigma_{n-1})\in \Sigma_{\unl{m}}}
  \frac{(-1)^{\sum_{i\in I} \ell(\sigma_i)|\alpha_i|}f(\{x_{i,\sigma_i(1)},\ldots,x_{i,\sigma_i(m_i)}\})}{\unl{m}!}.
\end{equation*}
This endows $\bar{\BW}^{V}$ with a structure of an associative unital algebra
with the unit $\textbf{1}\in \bar{\BW}^{V}_{(0,\ldots,0)}$.

We will be interested only in the submodule of $\bar{\BW}^{V}$ defined by
the \emph{pole} and \emph{wheel conditions}:

\noindent
$\bullet$
We say that $F\in \bar{\BW}^{V}_{\underline{k}}$ satisfies
the \emph{pole conditions} if
\begin{equation}\label{pole conditions super yangian}
  F=
  \frac{f(x_{1,1},\ldots,x_{n-1,k_{n-1}})}
       {\prod_{i=1}^{n-2}\prod_{r\leq k_i}^{r'\leq k_{i+1}}(x_{i,r}-x_{i+1,r'})},\
  f\in \BC[\hbar][\{x_{i,r}\}_{i\in I}^{1\leq r\leq k_i}],
\end{equation}
where the polynomial $f$ is supersymmetric in $\{x_{i,r}\}_{r=1}^{k_i}$ for all $i\in I$.

\noindent
$\bullet$
We say that $F\in \bar{\BW}^{V}_{\underline{k}}$ satisfies
the \emph{first kind wheel conditions} if
\begin{equation}\label{wheel condition 1 super yangian}
  F(\{x_{i,r}\})=0\ \mathrm{once}\
  x_{i,r_1}=x_{i+\epsilon,s}+\hbar/2=x_{i,r_2}+\hbar\
  \mathrm{for\ some}\ \epsilon, i, r_1, r_2, s,
\end{equation}
where
  $\epsilon\in \{\pm 1\}, i,i+\epsilon\in I,
   1\leq r_1,r_2\leq k_i, 1\leq s\leq k_{i+\epsilon}$
and $|\alpha_i|=\bar{0}$.

\noindent
$\bullet$
We say that $F\in \bar{\BW}^{V}_{\underline{k}}$ satisfies
the \emph{second kind wheel conditions} if
\begin{equation}\label{wheel condition 2 super yangian}
  F(\{x_{i,r}\})=0\ \mathrm{once}\
  x_{i-1,s}=x_{i,r_1}+\hbar/2=x_{i+1,s'}=x_{i,r_2}-\hbar/2
  \ \mathrm{for\ some}\  i,r_1,r_2,s,s',
\end{equation}
where
  $i,i-1,i+1\in I, 1\leq r_1,r_2\leq k_i, 1\leq s\leq k_{i-1}, 1\leq s'\leq k_{i+1}$
and $|\alpha_i|=\bar{1}$.

\medskip

Let $\bar{W}^{V}_{\underline{k}}\subset \bar{\BW}^{V}_{\underline{k}}$ denote the
$\BC[\hbar]$-submodule of all elements $F$ satisfying these three conditions and set
  $\bar{W}^{V}:=\underset{\underline{k}\in \BN^{I}}\bigoplus \bar{W}^{V}_{\underline{k}}$.
It is straightforward to check that $\bar{W}^{V}\subset \bar{\BW}^{V}$ is $\star$-closed.

\begin{Def}\label{Definition of Shuffle}
The algebra $\left(\bar{W}^{V},\star\right)$ shall be called the
\emph{rational shuffle (super)algebra}.
\end{Def}

%%%%%%%%%%%%%%%%%%%%%%% Deleted as it is moved to Section 2.7 now %%%%%%%%%%%%%%%%%%%
%Recall the $\BC$-subalgebra $Y^+(\ssl(V))$ of $Y(\ssl(V))$, generated by
%$\{\sx^+_{i,r}\}_{i\in I}^{r\geq 0}$, which is identified with the $\BC$-superalgebra
%$\wt{Y}^+(\ssl(V))$ generated by $\{\sx^+_{i,r}\}_{i\in I}^{r\geq 0}$ with the defining
%relations~(\ref{Dr Yangian sl 4}--\ref{Dr Yangian sl 6}), due to
%Proposition~\ref{Triangular decomposition}(a). Likewise, the $\BC[\hbar]$-subalgebra
%$Y^+_\hbar(\ssl(V))$ of $Y_\hbar(\ssl(V))$, generated by $\{\sx^+_{i,r}\}_{i\in I}^{r\geq 0}$,
%is isomorphic to the $\BC[\hbar]$-superalgebra generated by $\{\sx^+_{i,r}\}_{i\in I}^{r\geq 0}$
%subject to the defining relations~(\ref{Dr Yangian sl 4}--\ref{Dr Yangian sl 6},~\ref{Dr Yangian sl 4 formal})
%and with the $\BZ_2$-grading given by $|\sx^+_{i,r}|=|\alpha_i|$.
%%%%%%%%%%%%%%%%%%%%%%%%%%%%%%%%%%%%%%%%%%%%%%%%%%%%%%%%%%%%%%%%%%%%%%%%%%%%%%%%%%%%%

This algebra is related to $Y^+_\hbar(\ssl(V))$ of Section~\ref{ssec formal Yangian}
via the following construction:

\begin{Prop}\label{simple shuffle}
The assignment $\sx^+_{i,r}\mapsto x_{i,1}^r\ (i\in I,r\in \BN)$
gives rise to a $\BC[\hbar]$-algebra homomorphism
$\Psi\colon Y^+_\hbar(\ssl(V))\to \bar{W}^{V}$.
\end{Prop}

\begin{proof}
The assignment $\sx^+_{i,r}\mapsto x_{i,1}^r\ (i\in I,r\in \BN)$ is compatible
with the defining
relations~(\ref{Dr Yangian sl 4.2}--\ref{Dr Yangian sl 6},~\ref{Dr Yangian sl 4 formal})
of $Y^+_\hbar(\ssl(V))$, due to Theorem~\ref{PBW for formal Yangian}(c).
Hence, it gives rise to a $\BC[\hbar]$-algebra homomorphism
$\Psi\colon Y^+_\hbar(\ssl(V))\to \bar{W}^{V}$.
\end{proof}

The injectivity of $\Psi$ will be proved in Corollary~\ref{Injectivity of Psi},
while its image will be identified with the submodule $W^V$ of \emph{good} elements,
see Definition~\ref{good element yangian} and Theorem~\ref{hard shuffle yangian}
(in particular, the cokernel of $\Psi$ is an $\hbar$-torsion module), resulting
in the algebra isomorphism $Y^+_\hbar(\ssl(V))\iso W^{V}$.
This constitutes the first main result of this section.

\medskip

Recall the $\BC[\hbar]$-subalgebra $\bY^+_\hbar(\ssl(V))$ of $Y^+_\hbar(\ssl(V))$,
generated by $\{\SX^+_{\beta,r}\}_{\beta\in \Delta^+}^{r\in \BN}$
of~(\ref{normalized generators}). Our second key result of this section
provides an explicit description of the image $\Psi(\bY^+_\hbar(\ssl(V)))$.

\begin{Def}\label{integral element yangian}
$F\in \bar{W}^{V}_{\unl{k}}$ is \textbf{integral} if $F$
is divisible by $\hbar^{k_1+\ldots+k_{n-1}}$.
\end{Def}

Set
  $\fW^{V}:=\underset{\unl{k}\in \BN^I}\bigoplus \fW^{V}_{\underline{k}}$,
where $\fW^{V}_{\unl{k}}\subset \bar{W}^{V}_{\unl{k}}$ denotes the
$\BC[\hbar]$-submodule of all \emph{integral} elements.
The following is the second main result of this section:

\begin{Thm}\label{shuffle integral form yangian}
The $\BC[\hbar]$-algebra homomorphism
  $\Psi\colon Y^+_\hbar(\ssl(V))\to \bar{W}^{V}$
gives rise to a $\BC[\hbar]$-algebra isomorphism
  $\Psi\colon \bY^+_\hbar(\ssl(V))\iso \fW^{V}$.
\end{Thm}

As a corollary, we will also obtain a new proof of the following result
(cf.~Theorem~\ref{pbw for entire gavarini yangian}):

\begin{Thm}\label{pbw for gavarini yangian}
(a) The subalgebra $\bY^+_\hbar(\ssl(V))$ is independent of all our choices
in~(\ref{higher roots}).

\noindent
(b) The ordered PBW monomials $\{\SX^+_h\}_{h\in H}$ form
a basis of a free $\BC[\hbar]$-module $\bY^+_\hbar(\ssl(V))$.
\end{Thm}

    %%%%%%%%%%%%%%%%%%%%%%%%%%%%%%%%%%%%%%%%%%%%%%%%%%%%%%%%%%%%%%%%%%%%%%%%%
    %%%%%%%%%%%%%%%%%%% Rank 1 Computations  %%%%%%%%%%%%%%%%%%%%%%%%%%%%%%%%
    %%%%%%%%%%%%%%%%%%%%%%%%%%%%%%%%%%%%%%%%%%%%%%%%%%%%%%%%%%%%%%%%%%%%%%%%%

\subsection{The image of $Y^+_\hbar(\ssl(V))$ for $\dim(V)=2$}\label{ssec rank 1 case}
\

In the simplest case $\dim(V)=2$, all shuffle elements are
\emph{good} (Definition~\ref{good element yangian}), that is
$W^V=\bar{W}^{V}$ (see Remark~\ref{good for rk 1}(a)).
Therefore, Theorem~\ref{hard shuffle yangian} is equivalent to

\begin{Thm}\label{rank 1 iso}
If $\dim(V)=2$, then $\Psi\colon Y^+_\hbar(\ssl(V))\to \bar{W}^{V}$
is an algebra isomorphism.
\end{Thm}

There are two cases to consider:
  $\bar{1}\ne \bar{2}$ (so that $|\sx^+_{1,r}|=\bar{1}$)
and
  $\bar{1}=\bar{2}$ (so that $|\sx^+_{1,r}|=\bar{0}$).
First, assume $\bar{1}\ne \bar{2}$. Due to Theorem~\ref{PBW for formal Yangian},
the following result implies Theorem~\ref{rank 1 iso}:

\begin{Lem}\label{rank 1 odd}
The ordered products
  $\{x^{r_1}\star x^{r_2}\star\cdots \star x^{r_k}\}_{k\in \BN}^{0\leq r_1<\cdots<r_k}$
form a $\BC[\hbar]$-basis of $\bar{W}^V$.
\end{Lem}

\begin{proof}
This follows from the $\BC[\hbar]$-algebra isomorphism
$\bar{W}^V\simeq \bigoplus_k\Lambda_k$, where $\Lambda_k$ denotes the
$\BC[\hbar]$-module of skew-symmetric $\BC[\hbar]$-polynomials in $k$ variables,
while the algebra structure on the direct sum arises via the standard
skew-symmetrization maps $\Lambda_k\otimes \Lambda_l\to \Lambda_{k+l}$.
\end{proof}

Next, assume $\bar{1}=\bar{2}$. Due to Theorem~\ref{PBW for formal Yangian},
the following result implies Theorem~\ref{rank 1 iso}:

\begin{Lem}\label{rank 1 even}
The ordered products
  $\{x^{r_1}\star x^{r_2}\star\cdots \star x^{r_k}\}_{k\in \BN}^{0\leq r_1\leq \cdots\leq r_k}$
form a $\BC[\hbar]$-basis of $\bar{W}^V$.
\end{Lem}

\begin{proof}
Recall from~\cite[Lemma 6.22]{t} that the $k$-th power of
$x^r\in \bar{W}^V_1$ ($k\geq 1, r\geq 0$) equals
$x^r\star\cdots \star x^r=k\cdot (x_1\cdots x_k)^r$.
Therefore, for any ordered collection
\begin{equation*}
  0\leq r_1=\cdots=r_{k_1}<r_{k_1+1}=\cdots=r_{k_1+k_2}<
  \cdots < r_{k_1+\ldots+k_{l-1}+1}=\cdots=r_{k=k_1+\ldots+k_l},
\end{equation*}
it is clear that $x^{r_1}\star\cdots \star x^{r_k}$ is
a symmetric polynomial of the form
  $\nu_{\unl{r}} m_{(r_{1},\ldots,r_{k})}(x_1,\ldots,x_k)+
   \sum \nu_{\unl{r}'}m_{\unl{r}'}(x_1,\ldots,x_k)$.
Here, $m_{\unl{r}}(x_1,\ldots,x_k)$ are the monomial symmetric polynomials,
the sum is over $\unl{r}'=(r'_1\leq\cdots\leq r'_k)$ satisfying
$r_1\leq r'_1\leq r'_k\leq r_k$, $\nu_{\unl{r}'}\in \BC[\hbar]$,
and $\nu_{\unl{r}}=\prod_{i=1}^l k_i$.

This completes the proof of Lemma~\ref{rank 1 even} as
  $\{m_{(s_1,\ldots,s_k)}(x_1,\ldots,x_k)\}_{0\leq s_1\leq\cdots\leq s_k}$
form a $\BC[\hbar]$-basis of
  $\BC[\hbar][\{x_r\}_{r=1}^k]^{\Sigma_k}\simeq \bar{W}^{V}_k$.
\end{proof}

Combining Lemmas~\ref{rank 1 odd} and~\ref{rank 1 even}, we obtain
the proof of Theorem~\ref{rank 1 iso}.

    %%%%%%%%%%%%%%%%%%%%%%%%%%%%%%%%%%%%%%%%%%%%%%%%%%%%%%%%%%%%%%%%%%%%%%%%%
    %%%%%%%%%%%%% Specialization maps and Injectivity  %%%%%%%%%%%%%%%%%%%%%%
    %%%%%%%%%%%%%%%%%%%%%%%%%%%%%%%%%%%%%%%%%%%%%%%%%%%%%%%%%%%%%%%%%%%%%%%%%

\subsection{The specialization maps and the injectivity of $\Psi$}\label{ssec usual shuffle algebra}
\

For an ordered PBW monomial $\sx^+_h\ (h\in H)$, define its \emph{degree}
$\deg(\sx^+_h)=\deg(h)\in \BN^{\frac{n(n-1)}{2}}$ as a collection of
  $d_{\beta}:=\sum_{r\in \BN} h(\beta,r)\ (\beta\in \Delta^+)$
ordered with respect to the total ordering~(\ref{order}) on $\Delta^+$.
We consider the \emph{lexicographical ordering} on the collections
$\unl{d}=\{d_\beta\}_{\beta\in \Delta^+}$ of $\BN^{\frac{n(n-1)}{2}}$:
\begin{equation*}
  \{d_\beta\}_{\beta\in \Delta^+}<\{d'_\beta\}_{\beta\in \Delta^+}
  \ \mathrm{iff\ there\ is}\ \gamma\in \Delta^+\
  \mathrm{such\ that}\ d_\gamma>d'_\gamma\ \mathrm{and}\
  d_\beta=d'_\beta\ \mathrm{for\ all}\ \beta<\gamma.
\end{equation*}

In what follows, we shall need an explicit formula for $\Psi(\sx^+_{\beta,r})$:

\begin{Lem}\label{shuffle root elt}
For $1\leq j<i<n$ and $r\in \BN$, we have
\begin{equation*}
  \Psi(\sx^+_{\alpha_j+\alpha_{j+1}+\ldots+\alpha_i,r})=
  \hbar^{i-j}\frac{p(x_{j,1},\ldots,x_{i,1})}{(x_{j,1}-x_{j+1,1})\cdots (x_{i-1,1}-x_{i,1})},
\end{equation*}
where $p(x_{j,1},\ldots,x_{i,1})$ is a degree $r$ monomial, up to a sign.
\end{Lem}

\begin{proof}
Straightforward computation.
\end{proof}

For $\beta=\alpha_j+\alpha_{j+1}+\ldots+\alpha_i$, define $j(\beta):=j, i(\beta):=i$,
and let $[\beta]$ denote the integer interval $[j(\beta);i(\beta)]$. Consider
a collection of the intervals $\{[\beta]\}_{\beta\in \Delta^+}$ each taken with
a multiplicity $d_{\beta}\in \BN$ and ordered with respect to~(\ref{order})
(the order inside each group is irrelevant), denoted by
$\cup_{\beta\in \Delta^+} [\beta]^{d_\beta}$. Define $\unl{l}\in \BN^I$
via $\unl{l}:=\sum_{\beta\in \Delta^+} d_{\beta}[\beta]$.
Let us now define the \emph{specialization map}
\begin{equation}\label{specialization map}
  \phi_{\unl{d}}\colon \bar{W}^{V}_{\unl{l}}\longrightarrow
  \BC[\hbar][\{y_{\beta,s}\}_{\beta\in \Delta^+}^{1\leq s\leq d_\beta}].
\end{equation}
Split the variables $\{x_{i,r}\}_{i\in I}^{1\leq r\leq l_i}$ into
$\sum_{\beta\in \Delta^+} d_\beta$ groups corresponding to the above
intervals, and specialize the variable $x_{k,\ast}$ in the $s$-th copy of
$[\beta]$ to $y_{\beta,s}+\frac{c_{12}+\ldots+c_{k-1,k}}{2}\hbar$
(so that the $x_{\ast,\ast}$-variables in the $s$-th copy of the interval $[\beta]$
are specialized to various $\hbar$-shifts of the same new variable $y_{\beta,s}$).
For
  $F=\frac{f(x_{1,1},\ldots,x_{n-1,l_{n-1}})}
   {\prod_{i=1}^{n-2}\prod_{1\leq r\leq k_i}^{1\leq r'\leq k_{i+1}} (x_{i,r}-x_{i+1,r'})}
   \in \bar{W}^{V}_{\unl{l}}$,
we finally define $\phi_{\unl{d}}(F)$ as the corresponding specialization
of its numerator $f$.

\begin{Rem}
Note that $\phi_{\unl{d}}(F)$ is independent of our splitting of the variables
$\{x_{i,r}\}_{i\in I}^{1\leq r\leq l_i}$ into groups and is supersymmetric
in $\{y_{\beta,s}\}_{s=1}^{d_\beta}$ for each $\beta\in \Delta^+$
(recall $|\beta|=|\alpha_{j(\beta)}|+\ldots+|\alpha_{i(\beta)}|$).
\end{Rem}

The key properties of the \emph{specialization maps} $\phi_{\unl{d}}$
are summarized in the next two lemmas.

\begin{Lem}\label{lower degrees}
If $\deg(h)<\unl{d}$, then $\phi_{\unl{d}}(\Psi(\sx^+_h))=0$.
\end{Lem}

\begin{proof}
The above condition guarantees that $\phi_{\unl{d}}$-specialization of any
summand of the supersymmetrization appearing in $\Psi(\sx^+_h)$ contains among
all the $\zeta$-factors at least one factor of the form
$\zeta_{i,i+1}(-\frac{c_{i,i+1}}{2}\hbar)=0$, hence, it is zero.
The result follows.
\end{proof}

\begin{Lem}\label{same degrees}
The specializations
  $\{\phi_{\unl{d}}(\Psi(\sx^+_h))\}_{h\in H}^{\deg(h)=\unl{d}}$
are linearly independent over $\BC[\hbar]$.
\end{Lem}

\begin{proof}
Consider the image of
  $\sx^+_h=
   \prod\limits_{(\beta,r)\in \Delta^+\times \BN}^{\rightarrow} {\sx^+_{\beta,r}}^{h(\beta,r)}$
under $\Psi$. It is a sum of $(\sum_{\beta\in \Delta^+} d_{\beta})!$ terms,
and as in the proof of Lemma~\ref{lower degrees} most of them specialize
to zero under $\phi_{\unl{d}}$, $\unl{d}:=\deg(h)$. The summands which
do not specialize to zero are parametrized by
  $\Sigma_{\unl{d}}:=\prod_{\beta\in \Delta^+} \Sigma_{d_{\beta}}$.
More precisely, given
  $(\sigma_\beta)_{\beta\in \Delta^+}\in \Sigma_{\unl{d}}$,
the associated summand corresponds to the case when for all
  $\beta\in \Delta^+$ and $1\leq s\leq d_\beta$,
the $(\sum_{\beta'<\beta}d_{\beta'}+s)$-th factor of the
corresponding term of $\Psi(\sx^+_h)$ is evaluated at
  $\{y_{\beta,\sigma_\beta(s)}+
   \frac{c_{12}+\ldots+c_{k-1,k}}{2}\hbar\}_{j(\beta)\leq k\leq i(\beta)}$.
Similar to~\cite[Lemma 3.15]{t}, the image of this summand under
$\phi_{\unl{d}}$ may be written in the form
  $\prod_{\beta,\beta'\in \Delta^+}^{\beta<\beta'} G_{\beta,\beta'}
   \cdot \prod_{\beta\in \Delta^+}G_\beta
   \cdot \prod_{\beta\in \Delta^+} G_\beta^{(\sigma_\beta)}$
(up to a sign) with the factors
  $G_{\beta,\beta'}, G_\beta, G_\beta^{(\sigma_\beta)}$
to be specified below.

The factor $G_{\beta,\beta'}\ (\beta<\beta')$ arises as a product
of the specializations of the $\zeta$-factors (note that we ignore the
denominator $z$ in $\zeta_{k,k\pm 1}(z)$, but not in $\zeta_{k,k}(z)$)
among two variables, which are getting specialized to $\hbar$-shifts of
$y_{\beta,\ast}$ and $y_{\beta',\ast}$. Explicitly, we have
\begin{equation}\label{Factor 1}
\begin{split}
  & G_{\beta,\beta'}=
    \prod_{1\leq s\leq d_\beta}^{1\leq s'\leq d_{\beta'}} \prod_{k=j(\beta)}^{i(\beta)}
    \left\{(y_{\beta,s}-y_{\beta',s'})^{\delta_{k,j(\beta')-1}-\delta_{k,i(\beta')}}\times\right.\\
  & \left. (y_{\beta,s}-y_{\beta',s'}-(-1)^{\ol{k}}\hbar)^{\delta_{k-1\in [\beta']}}
           (y_{\beta,s}-y_{\beta',s'}+\delta_{|\alpha_k|,\bar{0}}(-1)^{\ol{k}}\hbar)^{\delta_{k\in [\beta']}}\right\}.
\end{split}
\end{equation}
In particular, the total power of $(y_{\beta,s}-y_{\beta',s'})$ in
$G_{\beta,\beta'}$ is nonnegative and equals
\begin{equation}\label{vanishing along diagonal}
  \#\{k| [\beta]\ni k\in [\beta'], |\alpha_k|=\bar{1}\}+
  \delta_{j(\beta)<j(\beta')}\delta_{i(\beta)+1\in [\beta']}.
\end{equation}

Likewise, the total factor $G_\beta\cdot G_\beta^{(\sigma_\beta)}$
arises as a product of:

\noindent
1) the specializations of $\Psi(\sx^+_{\beta,\ast})$,

\noindent
2) the specializations of the $\zeta$-factors (note that we ignore the
denominator $z$ in $\zeta_{k,k\pm 1}(z)$, but not in $\zeta_{k,k}(z)$) among
two variables, which are getting specialized to $\hbar$-shifts of $y_{\beta,\ast}$.

Due to Lemma~\ref{shuffle root elt}, the total contribution
of the specializations in 1) equals
\begin{equation}\label{combined factor 1}
  \hbar^{d_\beta(i(\beta)-j(\beta))}\cdot
  \prod_{s=1}^{d_\beta} p_{\beta,r_\beta(h,s)}(y_{\beta,\sigma_\beta(s)}),
\end{equation}
where the collection $\{r_\beta(h,1),\ldots,r_\beta(h,d_\beta)\}$ is obtained
by listing every $r\in \BN$ with multiplicity $h(\beta,r)>0$ in the non-decreasing
order and $p_{\beta,r}(y)$ are degree $r$ monic polynomials (obtained by evaluating
the monomials $p$ of Lemma~\ref{shuffle root elt} at $\hbar$-shifts of $y$).

\noindent
On the other hand, the total contribution of the specializations in 2) equals
\begin{equation}\label{combined factor 2}
\begin{split}
  & \prod_{1\leq s <s' \leq d_\beta}\prod_{j(\beta)<j\leq i(\beta)}
    \left((y_{\beta,\sigma_\beta(s)}-y_{\beta,\sigma_\beta(s')}-\hbar)
          (y_{\beta,\sigma_\beta(s)}-y_{\beta,\sigma_\beta(s')}+\hbar)\right)^{\delta_{|\alpha_j|,\bar{0}}}\times\\
  & \prod_{1\leq s <s' \leq d_\beta}\prod_{j(\beta)<j\leq i(\beta)}
    \left((y_{\beta,\sigma_\beta(s)}-y_{\beta,\sigma_\beta(s')}-(-1)^{\bar{j}}\hbar)
          (y_{\beta,\sigma_\beta(s)}-y_{\beta,\sigma_\beta(s')})\right)^{\delta_{|\alpha_j|,\bar{1}}}\times\\
  & \prod_{1\leq s <s' \leq d_\beta}
    \left(\frac{y_{\beta,\sigma_\beta(s)}-y_{\beta,\sigma_\beta(s')}+(-1)^{\overline{j(\beta)}}\hbar}
               {y_{\beta,\sigma_\beta(s)}-y_{\beta,\sigma_\beta(s')}}\right)^{\delta_{|\alpha_{j(\beta)}|,\bar{0}}}.
\end{split}
\end{equation}

While the product of the factors $G_\beta$ and $G_\beta^{(\sigma_\beta)}$
equals the product of expressions~(\ref{combined factor 1},~\ref{combined factor 2}),
we define each of them separately as follows:
\begin{equation}\label{Factor 2}
    G_{\beta}=
    \hbar^{d_\beta(i(\beta)-j(\beta))}\cdot
    \prod_{1\leq s\ne s'\leq d_{\beta}}
      (y_{\beta,s}-y_{\beta,s'})^{\lfloor \frac{\odd(\beta)}{2}\rfloor}
      (y_{\beta,s}-y_{\beta,s'}+\hbar)^{\even(\beta)+\lfloor \frac{\odd(\beta)-1}{2}\rfloor},
\end{equation}
\begin{equation}\label{Factor 3}
    G_\beta^{(\sigma_\beta)}=
    \prod_{s=1}^{d_\beta} p_{\beta,r_\beta(h,s)}(y_{\beta,\sigma_\beta(s)})\cdot
     \begin{cases}
       \prod_{s<s'} \frac{y_{\beta,\sigma_\beta(s)}-y_{\beta,\sigma_\beta(s')}+(-1)^{\ol{j(\beta)}}\hbar}
                    {y_{\beta,\sigma_\beta(s)}-y_{\beta,\sigma_\beta(s')}}, & \text{if } |\beta|=\bar{0}\\
       (-1)^{\sigma_\beta}, & \text{if } |\beta|=\bar{1}
     \end{cases},
\end{equation}
where
\begin{equation*}
  \even(\beta):=\#\{k\in [\beta]| |\alpha_k|=\bar{0}\}
  \ \mathrm{and}\
  \odd(\beta):=\#\{k\in [\beta]| |\alpha_k|=\bar{1}\}.
\end{equation*}
It is straightforward to verify that the products of~(\ref{Factor 2},~\ref{Factor 3})
and~(\ref{combined factor 1},~\ref{combined factor 2}) indeed coincide.

\medskip

Note that the factors
  $\{G_{\beta,\beta'}\}_{\beta<\beta'}\cup\{G_\beta\}_{\beta}$
of~(\ref{Factor 1},~\ref{Factor 2}) are independent of
  $(\sigma_\beta)_{\beta\in \Delta^+}\in \Sigma_{\unl{d}}$.
Therefore, the specialization $\phi_{\unl{d}}(\Psi(\sx^+_h))$
has the following form:
\begin{equation}\label{explicit formula for same degrees}
  \phi_{\unl{d}}(\Psi(\sx^+_h))=\pm
  \prod_{\beta,\beta'\in \Delta^+}^{\beta<\beta'} G_{\beta,\beta'}\cdot
  \prod_{\beta\in \Delta^+}G_\beta\cdot
  \prod_{\beta\in \Delta^+}
    \left(\sum_{\sigma_\beta\in \Sigma_{d_{\beta}}} G_\beta^{(\sigma_\beta)}\right).
\end{equation}

For $\beta\in \Delta^+$, consider a two-dimensional superspace
$V'_\beta$ with basis vectors $\sfv'_1$ and $\sfv'_2$ having the parity
$\ol{j(\beta)}$ and $\ol{i(\beta)}$, respectively. Then, the sum
  $\sum_{\sigma_\beta\in \Sigma_{d_{\beta}}} G_\beta^{(\sigma_\beta)}$
coincides with the value of the shuffle element
  $p_{\beta,r_\beta(h,1)}(x)\star\cdots \star p_{\beta,r_\beta(h,d_\beta)}(x)
   \in \bar{W}^{V'_\beta}_{d_\beta}$
evaluated at $\{y_{\beta,s}\}_{s=1}^{d_\beta}$. The latter elements are
linearly independent (they form a basis of $\bar{W}^{V'_\beta}_{d_\beta}$),
due to Lemmas~\ref{rank 1 odd} and~\ref{rank 1 even}.

Thus,~(\ref{explicit formula for same degrees}) together with the above
observation completes our proof of Lemma~\ref{same degrees}.
\end{proof}

\begin{Cor}\label{Injectivity of Psi}
The homomorphism $\Psi\colon Y^+_\hbar(\ssl(V))\to \bar{W}^V$ is injective.
\end{Cor}

\begin{proof}
Assume the contrary, that there is a nonzero $x\in Y^+_\hbar(\ssl(V))$
such that $\Psi(x)=0$. Due to Theorem~\ref{PBW for formal Yangian}, $x$
may be written in the form $x=\sum_{h\in H} c_h \sx^+_h$, where all but
finitely many of $c_h\in \BC[\hbar]$ are zero. Define
$\unl{d}:=\max\{\deg(h)|c_h\ne 0\}$. Applying the \emph{specialization map}
$\phi_{\unl{d}}$ to $\Psi(x)=0$, we get
  $\sum_{h\in H}^{\deg(h)=\unl{d}} c_h \phi_{\unl{d}}(\Psi(\sx^+_h))=0$
by Lemma~\ref{lower degrees}. Furthermore, we get $c_h=0$ for all $h\in H$
with $\deg(h)=\unl{d}$, due to Lemma~\ref{same degrees}.
This contradicts our choice of $\unl{d}$.

This completes our proof of the injectivity of $\Psi$.
\end{proof}

    %%%%%%%%%%%%%%%%%%%%%%%%%%%%%%%%%%%%%%%%%%%%%%%%%%%%%%%%%%%%%%%%%%%%%%%%%
    %%%%%%%%%%%%%%% Description of the image of Psi  %%%%%%%%%%%%%%%%%%%%%%%%
    %%%%%%%%%%%%%%%%%%%%%%%%%%%%%%%%%%%%%%%%%%%%%%%%%%%%%%%%%%%%%%%%%%%%%%%%%

\subsection{Proofs of the main results}\label{ssec image description}
\

In this section, we describe the images of $Y^+_\hbar(\ssl(V))$ and
$\bY^+_\hbar(\ssl(V))$ under $\Psi$. To present the former, we make
the following definition:

\begin{Def}\label{good element yangian}
$F\in \bar{W}^{V}_{\unl{k}}$ is \textbf{good} if $\phi_{\unl{d}}(F)$
is divisible by $\hbar^{\sum_{\beta\in \Delta^+} d_\beta(i(\beta)-j(\beta))}$
for any degree vector $\unl{d}=\{d_\beta\}_{\beta\in \Delta^+}$ such that
$\unl{k}=\sum_{\beta\in \Delta^+} d_\beta[\beta]$.
\end{Def}

Set
  $W^{V}:=\underset{\unl{k}\in \BN^I}\bigoplus W^{V}_{\underline{k}}$,
where $W^{V}_{\unl{k}}\subset \bar{W}^{V}_{\unl{k}}$ denotes the
$\BC[\hbar]$-submodule of all \emph{good} elements.

\begin{Rem}\label{good for rk 1}
(a) For $\dim(V)=2$, we have $W^V=\bar{W}^V$.

\noindent
(b) $\fW^V\subseteq W^V$ as any \emph{integral} shuffle element
is obviously \emph{good}.
\end{Rem}

\begin{Lem}\label{necessity for Yangian}
$\Psi(Y^+_\hbar(\ssl(V)))\subseteq W^{V}$.
\end{Lem}

\begin{proof}
The proof is completely analogous to that of~\cite[Lemma 6.19]{t}
for the particular case of~(\ref{distinguished Dynkin}).
For any $\beta\in \Delta^+,1\leq s\leq d_\beta$ and $j(\beta)\leq k<i(\beta)$,
$\zeta$-factors between the variables $x_{k,\ast}$ and $x_{k+1,\ast}$ that
are specialized to $\hbar$-shifts of $y_{\beta,s}$ always specialize
under $\phi_{\unl{d}}$ to a multiple of $\hbar$. It remains to note that
the total number of such pairs is exactly
$\sum_{\beta\in \Delta^+} d_\beta(i(\beta)-j(\beta))$.
\end{proof}

The following is the first key result of this section:

\begin{Thm}\label{hard shuffle yangian}
The $\BC[\hbar]$-algebra embedding
$\Psi\colon Y^+_\hbar(\ssl(V))\hookrightarrow \bar{W}^{V}$ gives rise to a
$\BC[\hbar]$-algebra isomorphism $\Psi\colon Y^+_\hbar(\ssl(V))\iso W^{V}$.
\end{Thm}

\begin{proof}
We need to show that any \emph{good} element $F\in W^{V}_{\unl{k}}$ belongs
to the submodule $M\cap W^{V}_{\unl{k}}$, where $M\subset W^{V}$ denotes
the $\BC[\hbar]$-submodule spanned by $\{\Psi(\sx^+_h)\}_{h\in H}$.
Let $T_{\unl{k}}$ denote a finite set consisting of all degree vectors
$\unl{d}=\{d_\beta\}_{\beta\in \Delta^+}\in \BN^{\frac{n(n-1)}{2}}$ such that
$\sum_{\beta\in \Delta^+} d_\beta [\beta]=\unl{k}$. We order $T_{\unl{k}}$ with
respect to the lexicographical ordering on $\BN^{\frac{n(n-1)}{2}}$. In particular,
the minimal element $\unl{d}_{\min}=\{d_\beta\}_{\beta\in \Delta^+}\in T_{\unl{k}}$
is characterized by $d_\beta=0$ for all non-simple roots $\beta\in \Delta^+$.

The proof is crucially based on the following result:

\begin{Lem}\label{spanning}
If $\phi_{\unl{d}'}(F)=0$ for all $\unl{d}'\in T_{\unl{k}}$ such that
$\unl{d}'>\unl{d}$, then there exists an element $F_{\unl{d}}\in M$
such that $\phi_{\unl{d}}(F)=\phi_{\unl{d}}(F_{\unl{d}})$ and
$\phi_{\unl{d}'}(F_{\unl{d}})=0$ for all $\unl{d}'>\unl{d}$.
\end{Lem}

\begin{proof}[Proof of Lemma~\ref{spanning}]
Consider the following total ordering on the set
$\{(\beta,s)\}_{\beta\in \Delta^+}^{1\leq s\leq d_\beta}$:
\begin{equation}\label{ordering of pairs}
  (\beta,s)\leq(\beta',s')\ \mathrm{iff}\
  \beta<\beta'\ \mathrm{or}\ \beta=\beta',s\leq s'.
\end{equation}

First, we note that the wheel
conditions~(\ref{wheel condition 1 super yangian},~\ref{wheel condition 2 super yangian})
for $F$ guarantee that $\phi_{\unl{d}}(F)$ (which is a polynomial in $\{y_{\beta,s}\}$)
vanishes up to appropriate orders under the following specializations:
\begin{enumerate}
\item[(i)] $y_{\beta,s}=y_{\beta',s'}+\hbar$ for $(\beta,s)<(\beta',s')$,

\item[(ii)]$y_{\beta,s}=y_{\beta',s'}-\hbar$ for $(\beta,s)<(\beta',s')$.
\end{enumerate}
\noindent
The orders of vanishing are computed similarly to~\cite[Remark 5.24]{t},
cf.~\cite{fhhsy,ne}. Let us view the specialization appearing in the definition
of $\phi_{\unl{d}}$ as a step-by-step specialization in each interval $[\beta]$.
As we specialize the variables in the new interval, we count only those wheel
conditions that arise from the non-specialized yet variables. Varying different
orderings of the intervals, we pick the maximal order of vanishing for each of
the linear terms $y_{\beta,s}-y_{\beta',s'}\pm\hbar$.
We claim that the resulting orders of vanishing under the specializations (i) and (ii)
exactly equal the powers of $y_{\beta,s}-y_{\beta',s'}\mp \hbar$ in $G_{\beta,\beta'}$
(if $\beta<\beta'$) or in $G_\beta$ (if $\beta=\beta'$). More precisely:

-- if $i(\beta)<i(\beta')$, then specializing first in the $s$-th copy of $[\beta]$
and then in the $s'$-th copy of $[\beta']$, the orders of vanishing under (i, ii)
equal the powers of $y_{\beta,s}-y_{\beta',s'}\mp \hbar$ in $G_{\beta,\beta'}$;

-- if $i(\beta)>i(\beta')$, then specializing first in the $s'$-th copy of $[\beta']$
and then in the $s$-th copy of $[\beta]$, the orders of vanishing under (i, ii)
equal the powers of $y_{\beta,s}-y_{\beta',s'}\mp \hbar$ in $G_{\beta,\beta'}$;

-- if $i(\beta)=i(\beta')$ (or $\beta'=\beta$), specializing first in the $s$-th copy
of $[\beta]$ and then in the $s'$-th copy of $[\beta']$, the orders of vanishing
under (i, ii) differ from the powers of $y_{\beta,s}-y_{\beta',s'}\mp \hbar$ in
$G_{\beta,\beta'}$ or $G_\beta$ (if $\beta'=\beta$) by an absence of a single factor
$y_{\beta,s}-y_{\beta',s'}+(-1)^{\ol{i(\beta)+1}}\hbar$. Reversing the order of specializations,
we end up missing only one factor $y_{\beta',s'}-y_{\beta,s}+(-1)^{\ol{i(\beta)+1}}\hbar$.
Hence, picking the maximal order of vanishing for each of $y_{\beta,s}-y_{\beta',s'}\mp \hbar$
achieves the result.

%%%%%%%%%%%%%%%%%%%%%%%%%% Alternative way to state the above last cae %%%%%%%%%%%%%%%%%%%%%%%%%%%%%%%%%%
%-- if $i(\beta)=i(\beta')$ (including the case $\beta'=\beta$), then the total contribution
%of the linear factors arising via (i, ii) as we specialize first in the $s$-th copy of
%$[\beta]$ and then in the $s'$-th copy of $[\beta']$ misses one factor
%$y_{\beta,s}-y_{\beta',s'}+(-1)^{\ol{i(\beta)+1}}\hbar$ in comparison to the corresponding
%factors in $G_{\beta,\beta'}$ or $G_\beta$. Reversing the order of specializations, we end
%up missing only one factor $y_{\beta',s'}-y_{\beta,s}+(-1)^{\ol{i(\beta)+1}}\hbar$. Hence,
%picking the maximal power of vanishing for both $y_{\beta,s}-y_{\beta',s'}\mp \hbar$ implies the result.
%%%%%%%%%%%%%%%%%%%%%%%%%%%%%%%%%%%%%%%%%%%%%%%%%%%%%%%%%%%%%%%%%%%%%%%%%%%%%%%%%%%%%%%%%%%%%%%%%%%%%%%%%

Second, we claim that $\phi_{\unl{d}}(F)$ vanishes under the following specializations:
\begin{enumerate}
\item[(iii)] $y_{\beta,s}=y_{\beta',s'}$ for $(\beta,s)<(\beta',s')$
             such that $j(\beta)<j(\beta')$ and $i(\beta)+1\in [\beta']$.
\end{enumerate}
Indeed, if $j(\beta)<j(\beta')$ and $i(\beta)+1\in [\beta']$, there are positive
roots $\gamma,\gamma'\in \Delta^+$ such that
  $j(\gamma)=j(\beta), i(\gamma)=i(\beta'),
   j(\gamma')=j(\beta'), i(\gamma')=i(\beta)$.
Consider the degree vector $\unl{d}'\in T_{\unl{k}}$ given by
  $d'_{\alpha}=d_{\alpha}+\delta_{\alpha,\gamma}+\delta_{\alpha,\gamma'}-
               \delta_{\alpha,\beta}-\delta_{\alpha,\beta'}$.
Then, $\unl{d}'>\unl{d}$ and thus $\phi_{\unl{d}'}(F)=0$. The result follows.

Finally, we also note that the skew-symmetry of the elements of $W^V$ with
respect to the variables $\{x_{k,\ast}\}$ with $|\alpha_k|=\bar{1}$ implies that
$\phi_{\unl{d}}(F)$ vanishes under the following specializations:

\begin{enumerate}
\item[(iv)] $y_{\beta,s}=y_{\beta',s'}$ for $(\beta,s)<(\beta',s')$ and vanishing
            order is $\#\{k| [\beta]\ni k\in [\beta'], |\alpha_k|=\bar{1}\}$.
\end{enumerate}
For $\beta<\beta'$ and any $s,s'$, combining (iii) and (iv), we see that the order
of vanishing of $\phi_{\unl{d}}(F)$ at $y_{\beta,s}=y_{\beta',s'}$ exactly equals the power of
$y_{\beta,s}-y_{\beta',s'}$ in $G_{\beta,\beta'}$ as computed in~(\ref{vanishing along diagonal}).
Similar, for $\beta'=\beta$ and $1\leq s<s'\leq d_\beta$, combining (iii) and (iv), we see that
the order of vanishing of $\phi_{\unl{d}}(F)$ at $y_{\beta,s}=y_{\beta,s'}$ equals the power of
$y_{\beta,s}-y_{\beta,s'}$ in $G_{\beta}$ of~(\ref{Factor 2}) plus one if $|\beta|=\bar{1}$.

Combining the above vanishing conditions for $\phi_{\unl{d}}(F)$ with $F$
being \emph{good}, we see that $\phi_{\unl{d}}(F)$ is divisible exactly by
the product $\prod_{\beta<\beta'} G_{\beta,\beta'}\cdot \prod_{\beta}G_\beta$
of~(\ref{Factor 1},~\ref{Factor 2}). Therefore, we have
\begin{equation}\label{woops 1}
  \phi_{\unl{d}}(F)=
  \prod_{\beta,\beta'\in \Delta^+}^{\beta<\beta'} G_{\beta,\beta'}\cdot
  \prod_{\beta\in\Delta^+}G_\beta\cdot G
\end{equation}
for some supersymmetric polynomial
\begin{equation}\label{woops 2}
  G\in \BC[\hbar][\{y_{\beta,s}\}_{\beta\in \Delta^+}^{1\leq s\leq d_\beta}]^{\Sigma_{\unl{d}}}
  \simeq \bigotimes_{\beta\in \Delta^+} \BC[\hbar][\{y_{\beta,s}\}_{s=1}^{d_\beta}]^{\Sigma_{d_\beta}},
\end{equation}
where $\BC[\hbar][\{y_{\beta,s}\}_{s=1}^{d_\beta}]^{\Sigma_{d_\beta}}$
denotes the submodule of symmetric (resp.\ skew-symmetric) polynomials in
$\{y_{\beta,s}\}_{s=1}^{d_\beta}$ if $|\beta|=\bar{0}$ (resp.\ $|\beta|=\bar{1}$).

Combining this observation with formula~(\ref{explicit formula for same degrees})
and the discussion after it, we see that there is a linear combination
$F_{\unl{d}}=\sum_{h\in H}^{\deg(h)=\unl{d}} c_h \sx^+_h$ such that
$\phi_{\unl{d}}(F)=\phi_{\unl{d}}(F_{\unl{d}})$, due to our proof of
Lemma~\ref{same degrees}. The equality $\phi_{\unl{d}'}(F_{\unl{d}})=0$
for $\unl{d}'>\unl{d}$ is due to Lemma~\ref{lower degrees}.

This completes our proof of Lemma~\ref{spanning}.
\end{proof}

Let $\unl{d}_{\max}$ and $\unl{d}_{\min}$ denote the maximal and the minimal
elements of $T_{\unl{k}}$, respectively. The condition of Lemma~\ref{spanning}
is vacuous for $\unl{d}=\unl{d}_{\max}$. Therefore, Lemma~\ref{spanning} applies.
Applying it iteratively, we eventually find an element $\wt{F}\in M$ such
that $\phi_{\unl{d}}(F)=\phi_{\unl{d}}(\wt{F})$ for all $\unl{d}\in T_{\unl{k}}$.
In the particular case of $\unl{d}=\unl{d}_{\min}$, this yields $F=\wt{F}$
(as the \emph{specialization map} $\phi_{\unl{d}_{\min}}$ essentially does not
change the function). Hence, $F\in M$.

This completes our proof of Theorem~\ref{hard shuffle yangian}.
\end{proof}

Using the same arguments, let us now prove Theorem~\ref{shuffle integral form yangian}.

\begin{proof}[Proof of Theorem~\ref{shuffle integral form yangian}]
The proof proceeds in two steps: first, we establish the inclusion
$\Psi(\bY^+_\hbar(\ssl(V)))\subseteq \fW^{V}$, and then we verify
the opposite inclusion $\Psi(\bY^+_\hbar(\ssl(V)))\supseteq \fW^{V}$.

\begin{Lem}\label{necessity for Gavarini Yangian}
$\Psi(\bY^+_\hbar(\ssl(V)))\subseteq \fW^{V}$.
\end{Lem}

\begin{proof}
According to Lemma~\ref{shuffle root elt}, $\Psi(\sX^+_{\beta,r})$
is divisible by $\hbar^{i(\beta)-j(\beta)+1}$. It remains to note that
\begin{equation}\label{shifted hbar as degree}
  \sum_{\beta\in \Delta^+} d_\beta(i(\beta)-j(\beta)+1)=\sum_{i\in I} l_i,
\end{equation}
where $\unl{l}=(l_1,\ldots,l_{n-1})\in \BN^I$ is defined via
$(l_1,\ldots,l_{n-1}):=\sum_{\beta\in \Delta^+} d_\beta[\beta]$.
\end{proof}

The proof of the opposite inclusion $\Psi(\bY^+_\hbar(\ssl(V)))\supseteq \fW^{V}$
is completely analogous to our proof of Theorem~\ref{hard shuffle yangian} and
Lemma~\ref{spanning}. Indeed, it suffices to note that the factor
$\hbar^{d_\beta(i(\beta)-j(\beta))}$ in the definition of $G_\beta$~(\ref{Factor 2})
shall be replaced by $\hbar^{d_\beta(i(\beta)-j(\beta)+1)}$, which
does not affect~(\ref{woops 1}), due to~(\ref{shifted hbar as degree})
(as we replaced ``$F$ being \emph{good}'' by ``$F$ being \emph{integral}'').
\end{proof}

We conclude this section with a new proof of Theorem~\ref{pbw for gavarini yangian}.

\begin{proof}[Proof of Theorem~\ref{pbw for gavarini yangian}]
(a) As $\Psi\colon Y^+_\hbar(\ssl(V))\to W^V$ is injective and the image
of $\bY^+_\hbar(\ssl(V))$, the submodule $\fW^V$, is independent of any choices
of $\sx^+_{\beta,r}$, Theorem~\ref{pbw for gavarini yangian}(a) follows.

\noindent
(b) Following the proofs of Theorems~\ref{shuffle integral form yangian} and~\ref{hard shuffle yangian},
we have already established that $\fW^V$ is $\BC[\hbar]$-spanned by the images of
the ordered PBW monomials $\{\Psi(\SX^+_h)\}_{h\in H}$. Combining this with
the injectivity of $\Psi$ and Theorem~\ref{PBW for formal Yangian}, completes the proof of
Theorem~\ref{pbw for gavarini yangian}(b).
\end{proof}

    %%%%%%%%%%%%%%%%%%%%%%%%%%%%%%%%%%%%%%%%%%%%%%%%%%%%%%%%%%%%%%%%%%%%%%%%%%%%%%%
    %%%%%%%%%%%%%%%%%%%%%%%%%%%%%%%%%%%%%%%%%%%%%%%%%%%%%%%%%%%%%%%%%%%%%%%%%%%%%%%
    %%%%%%%%%%%%%%%%%%%%%%%%%%%%% Quantum superalgebras %%%%%%%%%%%%%%%%%%%%%%%%%%%
    %%%%%%%%%%%%%%%%%%%%%%%%%%%%%%%%%%%%%%%%%%%%%%%%%%%%%%%%%%%%%%%%%%%%%%%%%%%%%%%
    %%%%%%%%%%%%%%%%%%%%%%%%%%%%%%%%%%%%%%%%%%%%%%%%%%%%%%%%%%%%%%%%%%%%%%%%%%%%%%%

\section{Generalization to type $A$ quantum affine superalgebras}
\label{sec shuffle for superquantum}

In this section, we briefly discuss the trigonometric counterparts of the
previous results. The quantum affine superalgebras were first studied $20$
years ago by Yamane~\cite{y}. In the~\emph{loc.cit.}, both the Drinfeld-Jimbo
and the new Drinfeld realizations were proposed and the isomorphism between
those was obtained. Also, the isomorphisms between the Drinfeld-Jimbo quantum
(affine) superalgebras corresponding to different Dynkin diagrams were constructed.

\begin{Rem}
%The explicit description of such isomorphisms in the new Drinfeld realization is rather nontrivial.
Such isomorphisms in the type $A$ toroidal setup, which does not admit the
Drinfeld-Jimbo realization, have been recently constructed in~\cite{bm}.
\end{Rem}

In this section, we obtain the shuffle algebra realizations of the
``positive halves'' of the quantum affine superalgebras of $\gl(V)$
corresponding to different Dynkin diagrams of $\gl(V)$.

Let $\vv$ be a formal variable. Define the ``positive half'' of the quantum
loop superalgebra of $\gl(V)$, denoted by $U^>_{\vv}(L\gl(V))$, to be the associative
$\BC(\vv)$-superalgebra generated by $\{e_{i,r}\}_{i\in I}^{r\in \BZ}$ with the
$\BZ_2$-grading $|e_{i,r}|=|\alpha_i|$, and subject to the following defining relations:
\begin{equation}\label{quantum 1}
  (z-\vv^{c_{ij}}w)e_i(z)e_j(w)=
  (-1)^{|\alpha_i|\cdot |\alpha_j|} (\vv^{c_{ij}}z-w)e_j(w)e_i(z),
\end{equation}
%%%%%%%%%%%%%%%%%%%%%%%%%%%%% COMMENT on the FIRST QUADRATIC RELATION %%%%%%%%%%%%%%%%%%%%%%%%%%%%%%%
% Similar to the Yangian case, we could require it unless $i=j$ and
%$|\alpha_i|=\bar{1}$, in which case it already follows from the next relation
%%%%%%%%%%%%%%%%%%%%%%%%%%%%%%%%%%%%%%%%%%%%%%%%%%%%%%%%%%%%%%%%%%%%%%%%%%%%%%%%%%%%%%%%%%%%%%%%%%%%%
\begin{equation}\label{quantum 2}
  [e_i(z),e_j(w)]=0 \ \mathrm{if}\ c_{ij}=0,
\end{equation}
as well as cubic $\vv$-Serre relations
\begin{equation}\label{quantum 3}
  [e_i(z_1),[e_i(z_2),e_j(w)]_{\vv^{-1}}]_\vv+
  [e_i(z_2),[e_i(z_1),e_j(w)]_{\vv^{-1}}]_\vv=0
  \ \mathrm{if}\ j=i\pm 1\ \mathrm{and}\ |\alpha_i|=\bar{0},
\end{equation}
and quartic $\vv$-Serre relations
\begin{equation}\label{quantum 4}
\begin{split}
  & [[[e_{i-1}(w),e_i(z_1)]_{\vv^{-1}},e_{i+1}(u)]_\vv,e_i(z_2)] +
    [[[e_{i-1}(w),e_i(z_2)]_{\vv^{-1}},e_{i+1}(u)]_\vv,e_i(z_1)]=0\\
  & \ \ \ \ \ \ \ \ \ \ \ \ \ \ \ \ \ \ \ \ \ \ \ \ \ \ \ \ \ \ \ \ \ \ \
    \ \ \ \ \ \ \ \ \ \ \ \ \ \ \ \ \ \ \ \ \ \ \ \ \ \ \
    \ \mathrm{if}\ |\alpha_i|=\bar{1}
    \ \mathrm{and}\ |\alpha_{i-1}|=|\alpha_{i+1}|=\bar{0},
\end{split}
\end{equation}
where $e_i(z):=\sum_{r\in \BZ} e_{i,r}z^{-r}$ and
$[a,b]_x:=ab-(-1)^{|a|\cdot |b|}x\cdot ba$ for homogeneous $a$ and $b$.

\begin{Rem}\label{v-Serre hold always}
The superalgebra $U^>_{\vv}(L\gl(V))$ is $\BN^I$-graded via
$\deg(e_{i,r})=1_i=(0,\ldots,1,\ldots,0)$ with $1$ at the $i$-th spot.
Given elements $a,b\in U^>_{\vv}(L\gl(V))$ with $\deg(a)=\unl{k}$ and
$\deg(b)=\unl{l}$, we set $(a,b):=\sum_{i,j\in I} k_il_jc_{ij}$.
Following~\cite[\S6.7]{y}, we define
  $[\![a,b]\!]:=ab-(-1)^{|a|\cdot |b|}\vv^{(a,b)}\cdot ba$.

\noindent
(a) The cubic $\vv$-Serre relations~(\ref{quantum 3}) can be written in the form
\begin{equation}\label{quantum 3 general}
  [\![e_i(z_1),[\![e_i(z_2),e_j(w)]\!]]\!] +
  [\![e_i(z_2),[\![e_i(z_1),e_j(w)]\!]]\!]=0.
\end{equation}
The relation~(\ref{quantum 3 general}) is also valid for $|\alpha_i|=\bar{1}$,
but in that case, it already follows from~(\ref{quantum 2}).

\noindent
(b) The quartic $\vv$-Serre relations~(\ref{quantum 4}) can be written in the form
\begin{equation}\label{quantum 4 general}
  [\![[\![[\![e_{i-1}(w),e_i(z_1)]\!],e_{i+1}(u)]\!],e_i(z_2)]\!] +
  [\![[\![[\![e_{i-1}(w),e_i(z_2)]\!],e_{i+1}(u)]\!],e_i(z_1)]\!]=0.
\end{equation}
The relation~(\ref{quantum 4 general}) is also valid for any other parities
of $\alpha_{i-1},\alpha_i,\alpha_{i+1}$, but in those cases, it already follows
from the quadratic and cubic relations~(\ref{quantum 1}--\ref{quantum 3}).
%%%%%%%%%%%%%%%%%%%%%%%%%%%%%%%%%%%%%%%%%%%%%%%%%%%%%%%%%%%%%%%%%%%%%%%%%%%%%%%%%%%%%%%%
%%%%%%%%%%%%%%%%%%%%%%%%%%%%%% Comment on this Quartic Relation %%%%%%%%%%%%%%%%%%%%%%%%
%%%%%%%%%%%%%%%%%%%%%%%%%%%%%%%%%%%%%%%%%%%%%%%%%%%%%%%%%%%%%%%%%%%%%%%%%%%%%%%%%%%%%%%%
% The best way to prove this is by verifying the corresponding equality on the shuffle side.
%
% Note that the cases $|\alpha_{i}|=\bar{0}$ and $|\alpha_i|=\bar{1}$ have to be treated separately.
%%%%%%%%%%%%%%%%%%%%%%%%%%%%%%%%%%%%%%%%%%%%%%%%%%%%%%%%%%%%%%%%%%%%%%%%%%%%%%%%%%%%%%%%
%%%%%%%%%%%%%%%%%%%%%%%%%%%%%%%%%%%%%%%%%%%%%%%%%%%%%%%%%%%%%%%%%%%%%%%%%%%%%%%%%%%%%%%%
%%%%%%%%%%%%%%%%%%%%%%%%%%%%%%%%%%%%%%%%%%%%%%%%%%%%%%%%%%%%%%%%%%%%%%%%%%%%%%%%%%%%%%%%

\noindent
(c) We finally note that~(\ref{quantum 4 general}) may be replaced by
the following equivalent relations:
\begin{equation}\label{quantum 4 general flv}
  [\![[\![e_{i-1}(w),e_i(z_1)]\!],[\![e_{i+1}(u),e_i(z_2)]\!]]\!] +
  [\![[\![e_{i-1}(w),e_i(z_2)]\!],[\![e_{i+1}(u),e_i(z_1)]\!]]\!] = 0.
\end{equation}
%%%%%%%%%%%%%%%%%%%%%%%%%%%%%%%%%%%%%%%%%%%%%%%%%%%%%%%%%%%%%%%%%%%%%%%%%%%%%%%%%%%%%%%%
%%%%%%%%%%%%%%%%%%%%%%%%%%%%%% Comment on this Quartic Relation %%%%%%%%%%%%%%%%%%%%%%%%
%%%%%%%%%%%%%%%%%%%%%%%%%%%%%%%%%%%%%%%%%%%%%%%%%%%%%%%%%%%%%%%%%%%%%%%%%%%%%%%%%%%%%%%%
% Again, it is new only in the case $|\alpha_i|=\bar{1},|\alpha_{i-1}|=|\alpha_{i+1}|=\bar{0}$,
% but it is true actually for any parities of $\alpha_{i-1},\alpha_i,\alpha_{i+1}$.
%
% One way to verify the latter is to use rewrite this equality on the shuffle side.
%
% Note that the cases $|\alpha_{i}|=\bar{0}$ and $|\alpha_i|=\bar{1}$ have to be treated separately.
%%%%%%%%%%%%%%%%%%%%%%%%%%%%%%%%%%%%%%%%%%%%%%%%%%%%%%%%%%%%%%%%%%%%%%%%%%%%%%%%%%%%%%%%
%%%%%%%%%%%%%%%%%%%%%%%%%%%%%%%%%%%%%%%%%%%%%%%%%%%%%%%%%%%%%%%%%%%%%%%%%%%%%%%%%%%%%%%%
%%%%%%%%%%%%%%%%%%%%%%%%%%%%%%%%%%%%%%%%%%%%%%%%%%%%%%%%%%%%%%%%%%%%%%%%%%%%%%%%%%%%%%%%
This is an affinization of the quartic $\vv$-Serre relations
of~\cite{flv} for finite quantum superalgebras.
\end{Rem}

Let us now define the trigonometric shuffle (super)algebra $\left(S^V,\star\right)$
analogously to the rational shuffle (super)algebra $\left(\bar{W}^V,\star\right)$ of
Section~\ref{ssec super shuffle algebra} with the following modifications:

\noindent
(1) All rational functions $F\in S^V$ are defined over $\BC(\vv)$;

\noindent
(2) The analogue of~(\ref{zeta-factor}) is the matrix
  $(\zeta_{i,j}(z))_{i,j\in I} \in \mathrm{Mat}_{I\times I}(\BC(\vv)(z))$
defined via
\begin{equation}\label{zeta quantum}
  \zeta_{i,j}(z)=
  (-1)^{\delta_{i>j}\delta_{|\alpha_i|,\bar{1}}\delta_{|\alpha_j|,\bar{1}}}
  (z-\vv^{-c_{ij}})/(z-1);
\end{equation}

\noindent
(3) The \emph{shuffle product} is defined via~(\ref{shuffle product}) with
$\zeta_{i,i'}(x_{i,r}-x_{i',r'})$ replaced by $\zeta_{i,i'}(x_{i,r}/x_{i',r'})$;

\noindent
(4) The \emph{pole conditions}~(\ref{pole conditions super yangian}) for
$F\in S^V_{\unl{k}}$ are modified as follows:
\begin{equation}\label{pole conditions yangian}
  F=
  \frac{f(x_{1,1},\ldots,x_{n-1,k_{n-1}})}
       {\prod_{i=1}^{n-2}\prod_{r\leq k_i}^{r'\leq k_{i+1}}(x_{i,r}-x_{i+1,r'})},
\end{equation}
where
  $f\in (\BC(\vv)[\{x^{\pm 1}_{i,r}\}_{i\in I}^{1\leq r\leq k_i}])^{\Sigma_{\unl{k}}}$
is a supersymmetric Laurent polynomial;

\noindent
(5) The \emph{first kind wheel conditions}~(\ref{wheel condition 1 super yangian})
for $F\in S^V$ are modified as follows:
\begin{equation}\label{wheel conditions quantum 1}
  F(\{x_{i,r}\})=0\ \mathrm{once}\
  x_{i,r_1}=\vv x_{i+\epsilon,s}=\vv^2x_{i,r_2}\
  \mathrm{for\ some}\ \epsilon, i, r_1, r_2, s,
\end{equation}
with $|\alpha_i|=\bar{0}$;

\noindent
(6) The \emph{second kind wheel conditions}~(\ref{wheel condition 2 super yangian})
for $F\in S^V$ are modified as follows:
\begin{equation}\label{wheel conditions quantum 2}
  F(\{x_{i,r}\})=0\ \mathrm{once}\ x_{i-1,s}=\vv x_{i,r_1}=x_{i+1,s'}=\vv^{-1} x_{i,r_2}
  \ \mathrm{for\ some}\  i, r_1,r_2,s,s',
\end{equation}
with $|\alpha_i|=\bar{1}$.

\medskip

The following is the main result of this section (announced in~\cite[\S8.2]{t}),
generalizing~\cite[Theorem 5.17]{t} for the particular case of~(\ref{distinguished Dynkin}):

\begin{Thm}\label{quantum shuffle isom}
The assignment $e_{i,r}\mapsto x_{i,1}^r\ (i\in I,r\in \BZ)$ gives rise
to an algebra isomorphism
\begin{equation}\label{q-Psi}
  \Psi\colon U^>_\vv(L\gl(V))\iso S^V.
\end{equation}
\end{Thm}

The proof of this theorem is completely analogous to our proof of
Theorem~\ref{hard shuffle yangian}.

\begin{Rem}
We note that~\cite[Theorems 3.34, 8.8]{t} providing the shuffle algebra realizations
of the RTT and Lusztig/Grojnowski/Chari-Pressley integral forms of $U^>_\vv(L\gl_n)$
can be straightforwardly generalized to the case of $U^>_\vv(L\gl(V))$. The former
has potential applications to the geometric representation theory
(cf.~\cite[Proposition 4.12, Remark 4.16]{ft2} for $n_-=0$).
\end{Rem}

    %%%%%%%%%%%%%%%%%%%%%%%%%%%%%%%%%%%%%%%%%%%%%%%%%%%%%%%%%%%%%%%%%%%%%%%%%
    %%%%%%%%%%%%%%%%%%%%%%%%%%%%%%%%%%%%%%%%%%%%%%%%%%%%%%%%%%%%%%%%%%%%%%%%%
    %%%%%%%%%%%%%%%%%%%%%%%%%%%%%% BIBLIOGRAPHY %%%%%%%%%%%%%%%%%%%%%%%%%%%%%
    %%%%%%%%%%%%%%%%%%%%%%%%%%%%%%%%%%%%%%%%%%%%%%%%%%%%%%%%%%%%%%%%%%%%%%%%%
    %%%%%%%%%%%%%%%%%%%%%%%%%%%%%%%%%%%%%%%%%%%%%%%%%%%%%%%%%%%%%%%%%%%%%%%%%

\end{document}